\documentclass{article}
\usepackage{geometry}
\usepackage{graphicx}  
\usepackage[dvipsnames]{xcolor}
\usepackage{amsmath,amsfonts,amssymb,amsthm}
\usepackage{dsfont}
\usepackage{mathrsfs}
\usepackage{mathabx}
\usepackage{comment}
\usepackage{authblk}
\usepackage[normalem]{ulem}
\usepackage{longtable}

\definecolor{royalblue}{rgb}{0.25, 0.41, 0.88}
\usepackage[pdftitle={},
pdfauthor={},
colorlinks=true,linkcolor=MidnightBlue,urlcolor=RoyalBlue,citecolor=royalblue,bookmarks=true,bookmarksopen=true,bookmarksopenlevel=2,unicode=true,linktocpage]{hyperref}
\usepackage[capitalise,noabbrev]{cleveref}

\usepackage[labelfont={bf,sf}, textfont=sf]{caption}
\usepackage{subcaption}

\numberwithin{equation}{section}

\makeatletter
\renewenvironment{abstract}{%
	\small
	\begin{center}%
		{\bfseries Abstract\vspace{-.5em}\vspace{\z@}}%
	\end{center}%
	\quotation}
{\endquotation}
\makeatother

\newcommand\bb{\mathbb}
\renewcommand\cal{\mathcal}
\renewcommand\r{\mathsf}

\newcommand{\ph}{\varphi}

\newcommand{\Tb}{\mathbb{T}}
\newcommand{\tfrak}{\mathfrak{t}}

\newcommand{\rhotilde}{\widetilde{\rho}}

\newcommand{\rh}{\mathsf{h}}
\newcommand{\rc}{\mathsf{c}}
\newcommand{\rH}{\mathsf{H}}
\newcommand{\rC}{\mathsf{C}}
\newcommand{\rF}{\mathsf{F}}

\def\Xint#1{\mathchoice
{\XXint\displaystyle\textstyle{#1}}%
{\XXint\textstyle\scriptstyle{#1}}%
{\XXint\scriptstyle\scriptscriptstyle{#1}}%
{\XXint\scriptscriptstyle\scriptscriptstyle{#1}}%
\!\int}
\def\XXint#1#2#3{{\setbox0=\hbox{$#1{#2#3}{\int}$ }
\vcenter{\hbox{$#2#3$ }}\kern-.6\wd0}}
\def\dashint{\Xint-}

\theoremstyle{plain}
\newtheorem{Thm}{Theorem}[section]
\newtheorem{Thm-fr}{Théorème}[section]
\newtheorem*{Thm*}{Theorem}
\newtheorem{Prop}[Thm]{Proposition}
\newtheorem*{Prop*}{Proposition}

\newtheorem*{Def*}{Definition}
\newtheorem{Lem}[Thm]{Lemma}
\newtheorem*{Cor*}{Corollary}
\newtheorem{Cor}[Thm]{Corollary}

\theoremstyle{definition}

\newtheorem*{Ex*}{Example}
\newtheorem{Rk}[Thm]{Remark}
\newtheorem*{Rmq*}{Remarques}

\begin{document}
\title{\textsc{
The self-dual point of Fortuin--Kasteleyn planar maps is critical}}
\date{}
\author{
Nathanaël Berestycki\thanks{University of Vienna, Austria, \texttt{nathanael.berestycki@univie.ac.at}} \quad \quad \quad 
William Da Silva\thanks{University of Vienna, Austria, \texttt{william.da.silva@univie.ac.at}}
}
\date{} 
\maketitle

\begin{abstract}
We study the Fortuin--Kasteleyn model of planar maps with parameter $q\in (0,4)$ at and away from its self-dual point. This model is also bijectively equivalent to the fully packed (bicoloured) loop-$O(n)$ model on planar triangulations.
These have been traditionally studied using either techniques from analytic combinatorics (based in particular on the gasket decomposition of Borot, Bouttier and Guitter \cite{BBG12a}) or probabilistic arguments (based on Sheffield's ``hamburger-cheeseburger'' bijection \cite{sheffield2016quantum}). In this paper we establish a dictionary relating quantities of interest in both approaches. This has several consequences. 
First, we derive an exact expression for the partition function of the fully packed (colour-symmetric) loop-$O(n)$ model on triangulations, as a function of the outer boundary length. This confirms predictions by Gaudin and Kostov \cite{gaudin1989n}. 
In particular, this model exhibits critical behaviour, in the sense that the partition function exhibits a power-law decay characteristic of the critical regime at this self-dual point. 
Secondly, we derive precise asymptotics for geometric features of the self-dual Fortuin--Kasteleyn model of planar maps when $0 < q <4$, such as the exact polynomial tail behaviour of the perimeters of clusters and loops. This sharpens previous results of \cite{berestycki2017critical} and \cite{gwynne2019scaling}. 
Finally, we prove that Fortuin--Kasteleyn maps undergo a \emph{sharp} phase transition at the self-dual point in the sense that, away from the self-dual point, cluster sizes decay exponentially. This mirrors the celebrated result by Beffara and Duminil-Copin for $q\geq 1$ on a fixed lattice \cite{beffara2012self} and constitutes the first result establishing that the self-dual point of Fortuin--Kasteleyn maps is the critical point. Along the way and as a key step of our proof, we use the above dictionary and the probabilistic results to justify rigorously an ansatz commonly assumed in the analytic combinatorics literature.
\end{abstract}

\tableofcontents

%
%
\section{Introduction}
\label{sec:intro}
Planar maps (i.e., proper embeddings of graphs in the $2$-sphere, considered up to orientation-preserving homeomorphisms) are a central topic not only in combinatorics but also in probability and mathematical physics due to their conjectured links with Liouville quantum gravity (LQG). In that context, planar maps can be thought of as canonical discretisations of the ``random surfaces'' which are at the core of Polyakov's original formulation of LQG \cite{polyakov1981quantum}. Indeed, in order to describe the gravitational action when gravity is coupled to a matter field, the so-called DDK Ansatz (\cite{David, DistlerKawai}) implies that this can be equivalently described by considering the scaling limit of random planar maps decorated by models of statistical mechanics at their critical point.

In this paper we will consider two such models, which turn out to be bijectively related. We start with the Fortuin--Kasteleyn percolation model with parameter $q\in (0,4)$. 
In this model, we sample a pair $(M, \Omega)$ where $M$ is a (rooted) planar map, and $\Omega$ is a subset of edges of $M$. In general, the probability of sampling a given decorated map $(\mathfrak m, \omega)$ depends on an extra percolation parameter $p_0 \in (0,1)$ and a volume weight $s$ and is proportional to 
\begin{equation}\label{eq:lawgen_s}
\bb P ((M, \Omega) = (\mathfrak m, \omega)) \,  \propto  \, 
\left(\frac{p_0}{1-p_0}\right)^{\r o(\omega)} q^{\r{cc}(\omega)}
(1/\sqrt{q})^{\r v(\mathfrak m)} s^{2\mathsf{e}(\mathfrak{m})},
\end{equation}
where $\r o(\omega)$ and $\r{cc}(\omega)$ are the number of edges and connected components of $\omega$, and $\mathsf{v} (\mathfrak m)$ and $\mathsf{e} (\mathfrak m)$ are the number of vertices and edges of $\mathfrak{m}$. 
The term $(1/\sqrt{q})^{\r v(\mathfrak m)}$ is for later convenience (see \cite[Equation (2.10)]{BBG12}).
We will mainly be focusing on a the model with fixed number of edges (say $k$ edges), which then becomes
\begin{equation}\label{eq:lawgen}
\bb P ((M, \Omega) = (\mathfrak m, \omega)) \,  \propto  \, 
\left(\frac{p_0}{1-p_0}\right)^{\r o(\omega)} q^{\r{cc}(\omega)} 
(1/\sqrt{q})^{\r v(\mathfrak m)}.
\end{equation}
In particular, conditional on the planar map $M = \mathfrak{m}$, the edge configuration $\Omega$ is sampled as an  FK$(q)$-percolation configuration on $\mathfrak{m}$ (see, e.g., \cite{duminil2013parafermionic} for a general introduction to this model). On such a decorated planar map $( \mathfrak m, \omega)$ there is a natural duality operation $( \mathfrak m, \omega) \to ( \mathfrak m^\dagger, \omega^\dagger)$. Requiring that the law \eqref{eq:lawgen} is invariant under this duality imposes that 
 $p_0 = p_{\mathrm{sd}}(q):=\sqrt{q}/(1+\sqrt{q})$, 
 as can be checked using Euler's formula.
In that case the law \eqref{eq:lawgen} reduces to 
\begin{equation}\label{eq:lawSD}
\bb P ((M, \Omega) = (\mathfrak m, \omega)) \,  \propto  \, 
\sqrt{q}^{\# \mathsf{loops}(\mathfrak m, \omega)},
\end{equation}
where $\mathsf{loops}(\mathfrak m, \omega)$ is the set of loops separating $\omega$ and its dual in $\mathfrak m$. This is called the \textbf{self-dual} FK$(q)$-weighted planar map model. 
{
In the self-dual case, one can also let the number of edges $k \to \infty$ to obtain, as a local limit of the above, an infinite FK$(q)$-weighted planar map model~\cite{sheffield2016quantum,chen2017basic}.
}

The second model we will discuss is that of the fully packed loop-$O(n)$ model on planar triangulations. 
In this model we sample a triangulation $T$, together with a fully packed  configuration $L$ of loops, with probability
\begin{equation}
    \label{eq:lawbicolour_sd}
    \mathbb{P} ((T, L) = ( \mathfrak{t}, \boldsymbol{\ell}) ) \; \propto \; x^{\# \mathsf{faces} (\mathfrak{t})} n^{\# \boldsymbol{\ell}}. 
\end{equation}
Up to considering the dual map (one may view the loops in $\boldsymbol{\ell}$ as crossing the triangles of $\mathfrak{t}$ or as a subset of edges of the dual map $\mathfrak{t}^\dagger$ of degree two) and conditioning on $\mathfrak{t}$, the loop configuration $\boldsymbol{\ell}$ is sampled as the classical loop-$O(n)$ model on $\mathfrak{t}^{\dagger}$, as introduced in \cite{DMNS81}.
Note, however, that \eqref{eq:lawbicolour_sd} has an extra weight $x$ that only depends on $\mathfrak{t}$ and accounts for the volume of the triangulation in the coupling $(\mathfrak{t},\boldsymbol{\ell})$; see below in \eqref{eq:honeycomb} for more details.
In fact, it is convenient to view the model \eqref{eq:lawbicolour_sd} as the symmetric (i.e., ``self-dual'') specialisation of a more general model where loops are assigned one of two possible colours, and there is a weight $x_1, x_2$ for each face crossed by a loop of colour $i = 1,2$. When $x_1 = x_2$ (the ``self-dual'' case) then this reduces to \eqref{eq:lawbicolour_sd}. This model, introduced in \cite{BBG12a}, is called the twofold loop-$O(n)$ model or bicoloured loop-$O(n)$ model; see \eqref{eq:lawbicolour} for its precise definition and Section \ref{sec:O(n)_model} for more explanations.

As is well known, and as will be recalled in Section \ref{sec:O(n)_model}, the two models \eqref{eq:lawgen} and \eqref{eq:lawbicolour} (and their self-dual versions, \eqref{eq:lawSD} and \eqref{eq:lawbicolour_sd}) are in measure-preserving bijection with one another, provided that 
\begin{equation}\label{eq:xc}
n = \sqrt{q} \quad ; \quad x = x_c=\frac{1}{ \sqrt{8(n+2)}} = \frac{1}{\sqrt{8( \sqrt{q}+2)}}. 
\end{equation}
Both models are believed to be \emph{critical} in some sense, as we now explain.

On the one hand, from the perspective of the self-dual FK model \eqref{eq:lawSD}, criticality (if it indeed holds) would agree with the fact that, in common with many planar models of statistical mechanics, the self-dual point coincides with its transition point (see for instance \cite{grimmett2018probability} for some background discussion).
In particular, such a statement is the content of a celebrated result of Beffara and Duminil-Copin \cite{beffara2012self} for the random cluster model (equivalently the Fortuin--Kasteleyn percolation model) on the square lattice and for $q\ge 1$.

From the perspective of the fully packed loop-$O(n)$ model the predictions are perhaps less clear.  It may be useful to make a comparison to the case of loop-$O(n)$ model on the \emph{honeycomb lattice} rather than planar maps. In the general case of this model (i.e., where the loop configuration is not assumed to be fully-packed),  a loop configuration $\boldsymbol{\ell}$ in a finite sub-domain of the honeycomb lattice is sampled proportionally to the weight 
\begin{equation}\label{eq:honeycomb}
z^{|\boldsymbol{\ell}|} n^{\# \boldsymbol{\ell}},
\end{equation}
where $z>0$ and $|\boldsymbol{\ell}|$ denotes the total length of the loops, i.e.\ the number of edges in the loop configuration. The parameter $z>0$ encodes the density of loops in this model. (Note however that, when we restrict the above measure to fully-packed loop configurations the formula looks superficially the same as \eqref{eq:lawbicolour_sd} with $x = z$, but the two models are in fact different: in \eqref{eq:lawbicolour_sd} the parameter $x>0$ should be viewed as a size parameter for the map, furthermore the fully-packed case of \eqref{eq:honeycomb} is obtained by sending $z \to \infty$).

On the hexagonal lattice, the loop-$O(n)$ model  \eqref{eq:honeycomb} is predicted to have a rich phase diagram (see e.g.\ \cite[Section 3]{peled2019lectures}). Notably, it was predicted by Nienhuis \cite{Nienhuis82, Nienhuis84} that there is a critical line, given by 
$$z = z_c(n) = 1/ \sqrt{2 + \sqrt{2-n}}$$ for $n \in (0,2]$ (and in fact $n \in [-2, 2]$ in the physics literature) separating the sub-critical phase ($z<z_c$) where the loops exhibit exponential decay and the critical phase ($z\geq z_c$) where the decay should only be polynomial. In fact, for fixed $n$, there should be two ``critical'' regimes {called dilute and dense respectively}, according as $z=z_c(n)$ or $z>z_c(n)$, 
where the model is expected to have different (but both conformally-invariant) scaling limits. 
See \cite[Section 5.6]{kager2004guide} for a statement of this conjecture.
Although this remains a famous open problem in the field, we also stress that spectacular progress has been made in various regions of the phase diagram -- we refer to the recent paper \cite{glazman2025planar} for the state of the art and an important breakthrough in this direction. It is not altogether clear whether the fully-packed case on the hexagonal lattice should correspond to a third phase with a distinct conformally invariant scaling limit or should be in the same universality class as the dense phase (see \cite{blote1994fully} for some discussion). Nonrigorous renormalisation arguments have suggested that the fully-packed case corresponds to an unstable fixed (critical) point, although on the square lattice, according to some predictions \cite{blote2012completely} the fully packed model appears to behave as a dense phase. 
In the setting of planar maps, there seems to be a consensus that the fully-packed case  belongs to the same universality class as the dense case and our results below support this. 
It would be interesting to study the fully packed model further on the hexagonal and square lattices, since typically, one expects the same behaviour on a fixed lattice and a random one when there is universality and/or conformal invariance.

Turning back to the predictions concerning planar maps, in both cases, one of the principal conjectures in the field -- closely related to the DDK ansatz mentioned above -- states that, when suitably renormalised and conformally embedded into the Riemann sphere, self-dual FK$(q)$-weighted planar maps and fully packed loop-$O(n)$ decorated planar maps converge to a $\gamma$-LQG surface, where 
\[
q = 2 + 2\cos\left(\frac{\pi \gamma^2}{2}\right)
\quad 
\text{and}
\quad 
\gamma\in(\sqrt{2},2).
\]
In addition, the loops separating primal and dual clusters of $\Omega$ are conjectured to converge jointly with the map to an independent Conformal Loop Ensemble (CLE) with parameter $\kappa' = 16/\gamma^2$.

This is supported by a considerable body of evidence. On the one hand, in the breakthrough paper \cite{sheffield2016quantum}, Sheffield provided a measure-preserving bijection between self-dual FK-weighted planar maps \eqref{eq:lawSD} and inventory accumulations, which in turn correspond to a pair of (non-Markovian) walks, and showed that in the scaling limit this pair of walks converges to a pair of correlated Brownian motions. In combination with the so-called Mating of Trees framework developed by Duplantier, Miller and Sheffield \cite{duplantier2021liouville} this result can be re-interpreted as a form of convergence towards a $\gamma$-LQG surface decorated with an independent space-filling SLE$_{\kappa'}$ curve, albeit for a relatively weak topology (the so-called ``peanosphere'' topology). Building on this foundational work, a number of observables (such as sizes of clusters and their boundaries) have been analysed and the associated exponents computed (see \cite{gwynne2019scaling,gwynne2017scaling,gwynne2015scaling} and \cite{berestycki2017critical}). These values are consistent with those that can be predicted by combining known results on the dimension of SLE$_\kappa$ (\cite{beffara2008dimension}) with the Knizhnik--Polyakov--Zamolodchikov (KPZ) identity, cf.\ \cite{duplantier2021liouville, HKPZ}. See, e.g., \cite[Chapter 4]{BP} for a discussion of these results and additional perspective.

On the other hand, 
a classical approach to the loop-$O(n)$ model is to make use of the spatial Markov property of the model. This results in the so-called \emph{gasket decomposition} which was in particular used by Borot, Bouttier and Guitter \cite{BBG12, BBG12a, BBG12b} {to establish the phase diagram of the loop-$O(n)$ model}. Building on powerful tools of analytic combinatorics they derived and solved (assuming a certain natural ansatz) an equation for the resolvent of the partition function of the model, which yields fine asymptotics. Such an approach has been made fully rigorous (including the proof of the above ansatz) by works of Budd and Chen \cite{budd2019peeling}, in the case of the so-called \emph{rigid} model on quadrangulations (where loops are constrained to enter and leave a given quadrangle through opposite edges). From this, Chen, Curien and Maillard \cite{ChenCurienMaillard} deduced some scaling limit results for the perimeter cascade of $O(n)$ loops, and Aïdékon, Da Silva and Hu \cite{aidekon2026scaling} proved the scaling limit for the volume of such {rigid} loop-$O(n)$ quadrangulations. A similar approach was also used by Borot, Duplantier and Guitter \cite{BBD16} to determine the nesting statistics in the \emph{bending energy} variant of the $O(n)$ model, {where loops can bend inside a quadrangle at a given energetic cost.} 

{However, we emphasise that the aforementioned ansatz (and thus, the phase diagram of the model) has so far only been rigorously established in the rigid case \cite{budd2019peeling}. The proof relies on special symmetries of bipartite maps and cannot be directly extended to the case of triangulations. We comment on the exact nature of the missing step in \cref{sss:gasket_dec}.}

\medskip One of the goals of this paper is to combine these two approaches. For instance, we obtain an exact relation between quantities naturally arising in Sheffield's bijection on the one hand, and partition functions for the loop-$O(n)$ model on the other (see  \cref{prop:prob_tau} for a precise statement). This works as a ``dictionary'' which allows us to translate results between the different points of view. As we will now see, this enables us to establish a number of consequences for both models. Our first main consequence is the complete proof ({including a proof of the above ansatz}) of an exact formula for the partition function $F_\ell$ of the fully packed loop-$O(n)$ model for a given boundary length $\ell \ge 1$, {defined as the sum of the loop-$O(n)$ weights over all triangulations with a boundary of length $\ell$.}
Here and in the rest of the paper, the expression $a_n \sim b_n$ means that  $a_n /b_n\to 1$ as $n\to \infty$.
Set 
\begin{equation} \label{eq:theta_n_relation}
    \theta = \frac{1}{\pi}\arccos\Big(\frac{n}{2}\Big) = \frac1\pi \arccos\Big(\frac{\sqrt{q}}{2}\Big). 
\end{equation}

\begin{Thm}[Expression and asymptotics for the partition function]
\label{thm:F_ell_main}
We have the exact expression
	\[
	F_\ell
	=
	\int_{\gamma_-}^{\gamma_+}
	\rho(y) y^{\ell} \mathrm{d}y,
    \quad \ell\geq 1,
	\] 
where 
    \[
    \rho(y) =  
    \frac{2^{-\theta - 1/2}}{\pi\theta(\gamma_+\!-\gamma_-)}  (\gamma_+-y)^{1-\theta} \bigg( \Big( \sqrt{2\gamma_+ - \gamma_- -y} + \sqrt{y-\gamma_-}\Big)^{2\theta} - \Big( \sqrt{2\gamma_+ - \gamma_- -y} - \sqrt{y-\gamma_-}\Big)^{2\theta}\bigg),
    \]
    and 
    \[
\gamma_+ = 2^{3/2} \cos\left( \frac{\pi\theta}{2}\right) \quad \text{and} \quad \gamma_+-\gamma_- = \frac{2^{3/2}}{\theta} \sin\left(\frac{\pi\theta}{2}\right).
\]
\end{Thm}

In particular, we deduce the following asymptotics.
\begin{Cor} \label{cor:asymp_F_ell_main}
    The partition function satisfies:
    \begin{equation} \label{eq:asymp_F_ell_main}
	F_\ell 
	\sim
	c \frac{\gamma_+^{\ell}}{\ell^{2-\theta}} \quad \text{as } \ell\to\infty,
	\end{equation}
	where $c =\frac{2^{\theta-1/2}}{\pi \theta} (\gamma_+-\gamma_-)^{\theta-1} \gamma_+^{1-\theta} \Gamma(2-\theta)$, and $\gamma_-$ and $\gamma_+$ are as above.
\end{Cor}

\noindent Observe that the form of the asymptotic \eqref{eq:asymp_F_ell_main} is characteristic of the critical (in fact, more precisely, \textbf{non-generic critical} according to the terminology of \cite{BBG12}) phase of the loop-$O(n)$ model, with the exponent $2-\theta$ interpolating between $3/2$ and $2$ (see e.g.\ \cite[Section 5]{BBG12}). More precisely, it lies in the so-called \emph{dense} phase of the model, where loops are conjectured to be non-simple and to touch each other in the scaling limit. 

Our second main result determines the exact tail exponents for the size of typical loops and clusters in the original FK$(q)$-weighted planar map model. This sharpens the main theorem of \cite{berestycki2017critical} by identifying the previously implicit $\ell^{o(1)}$ factor as a true constant {(a similar result can also be deduced from  \cite{gwynne2019scaling}, with the $\ell^{o(1)}$ term identified as a slowly varying function)}. We denote by $\mathfrak{L}$ and $\mathfrak{K}$ the typical loop and (filled-in) cluster in the \emph{infinite} FK$(q)$ planar map. These correspond to the local limit of a uniformly chosen loop or (filled-in) cluster in a finite FK$(q)$ map of size $k$, and then sending $k$ to infinity, see \cite[Theorem 1.1]{berestycki2017critical}.
We define the perimeter $|\mathfrak{L}|$ of the loop $\mathfrak{L}$ as the number of triangles it crosses, and the perimeter $|\partial \mathfrak{K}|$ of the cluster as the degree of its external face, see \cref{sec:typ_cluster} for more precise definitions.

\begin{Thm}[Exponents for loops and clusters]
\label{thm:exp_main}
    We have the following tail asymptotics: as $\ell\to\infty$,
    \[
    \mathbb{P}(|\partial \mathfrak{K}| = \ell) \sim \frac{C}{\ell^{3-2\theta}}
    \quad \text{and} \quad 
    \mathbb{P}(|\mathfrak{L}| = \ell) \sim \frac{C'}{\ell^{3-2\theta}},
    \]
    where $C$ and $C'$ are positive constants.
\end{Thm}

Our strategy to prove the above two theorems is the following. First, we make some connections between the gasket decomposition approach and Sheffield's hamburger-cheeseburger bijection. This allows us to express the partition function (up to an explicit factor) as a hitting time probability for the burger walk. We use this information in both directions, as it will enable us to: (a) rigorously justify the missing ansatz in the gasket decomposition approach, and then deduce the exact expression of the partition from there; (b) feed this information back into the burger walks to derive exact asymptotics for the hitting times, and thus for loops and clusters.

\medskip
The last main result of this paper concerns the Fortuin--Kasteleyn model away from self-duality. 
For this result we work under the \emph{free} model \eqref{eq:lawgen_s} with arbitrary number of edges.
Recall the self-dual parameter $p_{\mathrm{sd}}(q)= \sqrt{q}/(1+\sqrt{q})$, and define the two regions
\begin{equation} \label{eq:def_R_q}
\mathcal{R}_q := \Big\{(s,p_0) \in \mathbb{R}_+^* \; \Big| \;  s\leq x_c, \,  s\sqrt{\frac{p_0}{\sqrt{q}(1-p_0)}} \leq x_c \Big\} \quad \text{and} \quad 
\mathcal{S}_q := \mathcal{R}_q \setminus \{(x_c,p_{\mathrm{sd}}(q))\}.
\end{equation}
Although it is not clear at this point, we will see in \cref{sss:link_FK_O(n)} that the partition function of the FK($q$) model \eqref{eq:lawgen_s} with parameters $(s,p_0) \in \mathcal{R}_q$ is finite. Since the partition function of this model is finite, we may therefore normalise the weights \eqref{eq:lawgen_s} to get a probability distribution $\mathbb{P}_{(s,p_0)}$; we then call    $\mathfrak{K}$ the corresponding (filled-in) random cluster of the root edge, which may be of any colour $i=1,2$, corresponding to being primal or dual. (Note that in contrast with the notation used in \cref{thm:exp_main}, here the law $\mathbb{P}_{(s, p_0)}$ under which this random variable is considered is not that of any local limit.)
In addition, we interpret $\mathcal{S}_q$ as a \emph{subcritical} region of the FK($q$) model \eqref{eq:lawgen_s}, since the condition defining $\mathcal{S}_q$ in \eqref{eq:def_R_q} guarantees that the model stays away from its self-dual point $(x_c,p_{\mathrm{sd}}(q))$.
We note that since $\frac{p_{\mathrm{sd}}(q)}{1-p_{\mathrm{sd}}(q)}=\sqrt{q}$, the region $\mathcal{R}_q$ contains in particular all pairs $(s,p_0)$ with $s\leq x_c$ and $p_0 \leq p_{\mathrm{sd}}(q)$ 
(in addition, note that the region $\mathcal{S}_q$ also contains values of $(s,p_0)$ with $p_0>p_{\mathrm{sd}}(q)$ by taking $s$ small enough).
We let $K_{\ell}^{(i)}$, $i=1,2$, be the partition function of Fortuin--Kasteleyn maps such that the root cluster has colour $i$ and perimeter $\ell$. 
Our final main theorem shows the exponential decay of $K_{\ell}^{(i)}$ for $i=1,2$ in the subcritical phase $\mathcal{S}_q$.
\begin{Thm}[Exponential decay of $K_{\ell}^{(i)}$]
\label{thm:main_offsd}
    Consider the $\mathrm{FK}(q)$ model described in \eqref{eq:lawgen_s} and assume that $(s,p_0)\in \mathcal{S}_q$. Then, for all $i=1,2$ there exists a constant $C_i>0$ and $\gamma_i \in (0,1)$ such that for all $\ell \geq 1$, 
    \[
    K_{\ell}^{(i)} \leq C_i \gamma_i^{\ell}.
    \]
    In contrast, when $(s,p_0)\in\mathcal{R}_q \setminus \mathcal{S}_q$, the partition functions $K_{\ell}^{(i)}$, $i=1,2$, decay polynomially: more specifically, there is a constant $C>0$ such that 
    \[
    K_\ell^{(1)} = K_{\ell}^{(2)} \sim C \ell^{2\theta-4} \quad \text{as } \ell\to\infty,
    \]
    with $\theta$ given by \eqref{eq:theta_n_relation}.
\end{Thm}
\noindent

As an immediate consequence of this result we deduce the following: if $(s, p_0) \in \mathcal{S}_q$ then $\mathbb{P}_{(s, p_0)} ( |\partial \mathfrak{K}| = \ell)$ decays exponentially fast as $\ell \to \infty$, whereas it decays polynomially fast (with the stated exponent) when $(s, p_0) \in \mathcal{R}_q \setminus \mathcal{S}_q$.  
\cref{thm:main_offsd} is thus a statement which shows that the Fortuin--Kasteleyn model of planar maps undergoes a phase transition at the self-dual point, and that this phase transition is furthermore \emph{sharp} in the language of percolation: while, at the self-dual point, cluster sizes decay polynomially, away from self-duality they are exponentially small. 
On the square lattice, such a result was first established (for $q\geq 1$) by Beffara and Duminil-Copin~\cite{beffara2012self} (see also~\cite{duminil2019sharp}).
To the best of our knowledge, our paper provides the first result showing such a phase transition (and sharpness) for the FK($q$) model on planar maps.

There is however an interesting difference with the case of the infinite planar square lattice. In that case, if $p$ is chosen away from the self-dual point, it is well known that there is either percolation of the primal or dual configuration.  By contrast, note that in Theorem \ref{thm:main_offsd} above, if $(s, p_0) \in \mathcal{S}_q$ then there is exponential decay of \emph{both} $K^{(1)}_\ell$ and $K^{(2)}_\ell$, in particular irrespective of whether $p_0< p_{\mathrm{sd}}(q)$ or $p_0> p_{\mathrm{sd}}(q)$.
This suggests that there may not exist a canonical way to make sense of a local limit of these maps away from the self-dual point, even when one imposes the colour of $\mathfrak{K}$ to be the ``dominant'' one (i.e.\ primal when $p_0>p_{\mathrm{sd}}(q)$ and dual when $p_0 < p_{\mathrm{sd}}(q)$).

Finally, we stress that \eqref{thm:main_offsd} can actually be phrased in terms of the bicoloured loop-$O(n)$ model presented in more detail in \cref{sec:O(n)_model}. In that framework, the region $\mathcal{S}_q$ corresponds to the set of edge parameters that satisfy $x_1,x_2 \leq x_c$ and $(x_1, x_2) \neq (x_c,x_c)$. 
See \cref{sss:link_FK_O(n)} and the proof in \cref{sec:_offsd} for more details.

\medskip
The paper is organised as follows. \cref{sec:prelim} is dedicated to preliminaries on self-dual FK planar maps, the $O(n)$ model and the Mullin--Bernardi--Sheffield bijection. In \cref{sec:self_dual_critical}, we make some connections between Sheffield's hamburger-cheeseburger walk and the gasket decomposition of Borot, Bouttier and Guitter. In particular, this allows us to prove an ansatz on the \emph{cut} of the associated resolvent function using solely hamburger-cheeseburger arguments: we will derive an explicit formula for the right endpoint $\gamma_+$ of the cut.  We will then use this information in \cref{sec:resolvents}, where we solve the resolvent equation for the fully packed $O(n)$ model on triangulations, thus proving \cref{thm:F_ell_main}. Then, in \cref{sec:loop-clusters}, we come back to the hamburger-cheeseburger model and deduce \cref{thm:exp_main} by plugging back the information on the partition function.
Finally, we prove \cref{thm:main_offsd} in \cref{sec:_offsd}.

We stress that this dictionary has been used in the concurrent work \cite{da2025scaling}, which concerns the critical case $q=4$ (equivalently, $n=2$). In that case, the authors prove the analogue of \cref{thm:F_ell_main} by directing solving the gasket decomposition equation (without any extra input) and then deduce the scaling limit of FK$(4)$-weighted planar maps in the peanosphere sense, resolving the remaining open regime in Sheffield's work \cite{sheffield2016quantum}.

\paragraph{Acknowledgments.}
We thank Xingjian Hu, Ellen Powell, Xin Sun, Joonas Turunen and Mo Dick Wong for insightful discussions. 
We are especially grateful to Sasha Glazman for discussions concerning the literature on the loop $O(n)$ model on the hexagonal lattice, and helping us clarify some confusions regarding this in a preliminary version of the paper.
N.B.  acknowledges the support from the Austrian Science Fund (FWF) grants 10.55776/F1002 on ``Discrete random structures: enumeration and scaling limits" and 10.55776/PAT1878824 on ``Random Conformal Fields''.
W.D.S.\ is supported by the Austrian Science Fund (FWF) grant on “Emergent branching structures in random geometry” (DOI: 10.55776/ESP534).

%
%

%
%
\section{Preliminaries}
\label{sec:prelim}


\subsection{Self-dual FK-weighted random planar maps}

\subsubsection{Definition of the model} 
\label{sss:def_FK}

We begin by recalling the basic objects. For each integer $k>0$, let $\mathcal{M}_k$ denote the collection of rooted planar maps with exactly $k$ edges.  By this we mean embeddings of finite, connected graphs with $k$ edges into the sphere, considered up to orientation-preserving homeomorphism, together with a distinguished oriented edge called the \emph{root}.  For $\mathfrak{m}\in \mathcal{M}_k$, write $V(\mathfrak{m})$, $E(\mathfrak{m})$, and $F(\mathfrak{m})$ for its vertex, edge, and face sets.  
The \emph{root vertex} is the initial endpoint of the root edge, and the \emph{root face} is the face incident to the right side of the root edge.  
For a face $\mathfrak{f}\in F(\mathfrak{m})$ we define its degree to be the number of edge-sides of $\mathfrak{m}$ incident to $\mathfrak{f}$; in particular, if an edge lies entirely inside $\mathfrak{f}$, we count it twice.

Our goal is to introduce the law of an FK-decorated planar map of size $k$, namely a pair $(\mathfrak{m},\omega)$ with $\mathfrak{m}\in \mathcal{M}_k$ and $\omega\subset E(\mathfrak{m})$ representing the set of \emph{open} edges.  Rather than specifying the law directly on $(\mathfrak{m},\omega)$, it will be convenient to pass through a canonical triangulation built from this pair.  This construction is often referred to as the \textbf{Tutte map}; see \cref{fig:Tutte}.

\paragraph{Duality and the Tutte map.}
Given $(\mathfrak{m},\omega)$, let $\mathfrak{m}^\dagger$ be the planar \textbf{dual} of $\mathfrak{m}$, defined as follows: each dual vertex $v^\dagger$ corresponds to a face $\mathfrak{f}\in F(\mathfrak{m})$, and two dual vertices are joined by a dual edge whenever the corresponding faces are adjacent.  
We orient the dual root edge so that it crosses the primal root edge from right to left.

The set $\omega$ induces a complementary set $\omega^\dagger\subset E(\mathfrak{m}^\dagger)$ by declaring
\[
e^\dagger\in\omega^\dagger \quad\Longleftrightarrow\quad e\notin\omega.
\]
We now create an auxiliary quadrangulation $Q(\mathfrak{m})$ whose vertex set is $V(\mathfrak{m})\cup V(\mathfrak{m}^\dagger)$.  For each face $\mathfrak{f}$ of $\mathfrak{m}$, with corresponding dual vertex $v^\dagger$, we connect $v^\dagger$ by an edge to each primal vertex incident to $\mathfrak{f}$.  
Every primal edge $e$ together with its dual $e^\dagger$ produces exactly four such connecting edges, so that each face of $Q(\mathfrak{m})$ has degree four.  
The root edge of $Q(\mathfrak{m})$ is taken to point from the dual root vertex to the primal root vertex.

Into $Q(\mathfrak{m})$ we now insert the primal open edges $\omega$ (drawn in blue) and the dual open edges $\omega^\dagger$ (drawn in red).  
Each quadrilateral of $Q(\mathfrak{m})$ is bisected by exactly one of these edges into two \textbf{companion triangles}, so the resulting planar map, denoted by $T(\mathfrak{m},\omega)$ is a rooted triangulation.  
We refer to a triangle in $T(\mathfrak{m},\omega)$ as \emph{primal} or \emph{dual} depending on whether it arises from a primal or dual bisecting edge.

\begin{figure*}[t!]
    \medskip
    \centering
    \begin{subfigure}[t]{0.5\textwidth}
        \centering
        \includegraphics[page=1,width=1\textwidth]{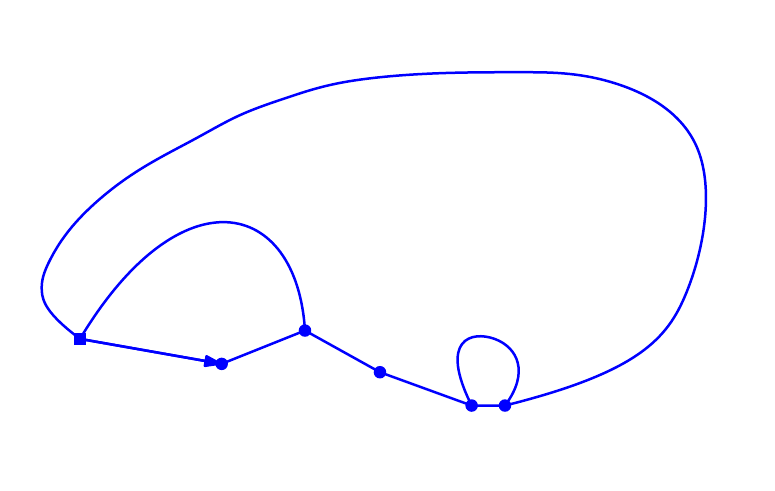}
        \caption{}
    \end{subfigure}%
    ~ 
    \begin{subfigure}[t]{0.5\textwidth}
        \centering
        \includegraphics[page=2,width=1\textwidth]{Example-Tutte.pdf}
        \caption{}
    \end{subfigure}%

    \vspace{0.5cm}
    \begin{subfigure}[t]{0.5\textwidth}
        \centering
        \includegraphics[page=3,width=1\textwidth]{Example-Tutte.pdf}
        \caption{}
    \end{subfigure}
    \caption{Construction of the Tutte map. 
    \textbf{(a)} The planar map $\mathfrak{m}$, with its oriented root edge.
    \textbf{(b)} We take the set of open edges $\omega$ (in blue) to be all the edges but one, and draw its dual $\omega^{\dagger}$ (in red). Then, we draw an edge (dashed) between each face (i.e.\ dual vertex) and any incident primal vertex. Considering only the dashed edges gives the quadrangulation $Q(\mathfrak{m})$, and keeping all the (blue, red and black) edges gives the Tutte triangulation $T(\mathfrak{m},\omega)$.
    \textbf{(c)} We colour the triangles blue and red according to the type of their (unique) coloured edge. We did not colour the infinite triangular face of the map (exterior to the drawing) which is also a blue triangle. The root triangle (dark blue) is the triangle to the right of the root edge of $T(\mathfrak{m},\omega)$. We draw the loops separating primal and dual components of the maps in purple. 
    \label{fig:Tutte}
    }
\end{figure*}

Consider any triangle of $T(\mathfrak{m},\omega)$ and traverse it through the edges that come from the quadrangulation $Q(\mathfrak{m})$ only.  
By following such steps whenever one enters a new triangle, one eventually returns to the starting point, producing a closed path.  
If we iterate this procedure until all the triangles are visited, we obtain a \textbf{fully packed} collection of disjoint simple loops
$L(\mathfrak{m},\omega)$.  
These loops mark the interfaces separating the primal and dual clusters induced by $(\omega,\omega^\dagger)$ (see again~Figure \ref{fig:Tutte}).

\paragraph{FK$(q)$-decorated maps.}
Fix $q\in (0,4)$.  
A self-dual \textbf{FK$(q)$ planar map with $k$ edges} is a random pair $(M,\Omega)$ with $M\in\mathcal{M}_k$ (recall that this denotes the set of planar maps with $k$ edges) and $\Omega\subset E(M)$ whose distribution is given by
\begin{equation}\label{def: FK sphere}
    \mathbb{P}\bigl((M,\Omega)=(\mathfrak{m},\omega)\bigr)
    \;\propto\; q^{\#L(\mathfrak{m},\omega)/2}.
\end{equation}
The weight depends only on the loop ensemble $L(\mathfrak{m},\omega)$ and therefore is independent of the choice of root edge; in other words, the root of $M$ is  sampled uniformly once the map is fixed.  
Conditionally on $M$, the edge configuration $\Omega$ is exactly the self-dual FK$(q)$ (or random-cluster) measure on $M$; see, for instance, \cite{duminil2013parafermionic}.


\subsubsection{The Mullin--Bernardi--Sheffield bijection}
\label{sec:mullin-bernardi-sheffield}

We summarise here the bijective constructions of Mullin \cite{mullin67}, Bernardi \cite{bernardi08} and Sheffield \cite{sheffield2016quantum}; see also \cite[Chapter 4]{BP} for another presentation of these ideas. These constructions relate planar maps equipped with a distinguished subset of edges to words written in the alphabet $\Theta$ introduced below. The latter may be interpreted as trajectories of a two-type inventory system evolving in discrete time.

\paragraph{The inventory accumulation model.}
Let $\Theta=\{\rc,\rh,\rC,\rH,\rF\}$
an alphabet of symbols.  A word is a finite concatenation $w=\theta_{1}\cdots\theta_{k}$ with $\theta_i\in\Theta$, the empty word being denoted by $\emptyset$.

For later intuition we call the letters $\rh$ and $\rc$ \emph{hamburgers} and \emph{cheeseburgers}, and we think of $\rC$ and $\rH$ as the corresponding \emph{orders}.  The symbol $\rF$ stands for a \emph{freshest} order.  A word is therefore read from left to right as a day in the life of a restaurant (i.e.\ a sequence of production events and customer requests) which forms a last-in--first-out kitchen selling two types of items.

The \textbf{reduction} of a word $w$ is obtained by repeatedly applying the relations
\[
\rc\rC=\rh\rH=\rc\rF=\rh\rF=\emptyset,\qquad
\rc\rH=\rH\rc,\qquad
\rh\rC=\rC\rh.
\]
Formally, the reduced word $\overline{w}$ is the equivalence class of $w$ under these relations, although in practice we will identify it with a given representative. 
It is useful to interpret this operation in the kitchen image. The relations are effectively pairing each order with the latest available matching burger: $\rC$ and $\rH$ consume, respectively, the most recent cheeseburger or hamburger, while $\rF$ consumes whichever of the two types lies topmost on the current stack of (available) burgers.  Under this interpretation $\overline{w}$ contains precisely orders that have remained unfulfilled, and those burgers which have not been ordered.

\paragraph{From decorated maps to words.}
We now explain how a pair $(\mathfrak m,\omega)$, where $\mathfrak m\in\mathcal{M}_k$ and $\omega\subset E(\mathfrak m)$, gives rise to a word $w\in\Theta^{2k}$.  The construction is  based on the Tutte triangulation $T(\mathfrak m,\omega)$ associated with $\mathfrak{m}, \omega$, which was defined in \cref{sss:def_FK}.  The general principle is to describe a certain space-filling exploration, in which primal (resp.\ dual) triangles yield $\rh,\rH$ (resp.\ $\rc,\rC$), while  $\rF$ symbols will correspond to a certain switching operation allowing to break up clusters.
{We advise the reader to follow the construction on \cref{fig:sheffield_exploration}.}

We begin with the set $L(\mathfrak m,\omega)$ of loops in $T(\mathfrak m,\omega)$.  Among these there is a unique loop $\mathfrak{l}_{0}$ that crosses the oriented root edge of the Tutte map; we orient it  so that it crosses that root edge from left to right.  This orders the triangles visited by $\mathfrak{l}_{0}$. To combine it with the other loops into a single space-filling path, we proceed as follows. 

Find the \emph{last} triangle $t$ visited by $\mathfrak{l}_{0}$ whose companion $t'$ lies on a distinct loop $\mathfrak{l}_{1}$.  Replace the diagonal of the quadrilateral $\{t,t'\}$ by its opposite one.  This merges $\mathfrak{l}_{0}$ and $\mathfrak{l}_{1}$ into a longer loop; such diagonals will be referred to as \emph{fictional edges} to remind ourselves that we were not present in the original edge configuration $\omega$.  Repeating this procedure inductively, all loops in $L(\mathfrak m,\omega)$ are eventually glued into a single loop $\mathfrak{l}$ traversing every triangle exactly once.  
Each time a diagonal is flipped during this gluing operation, an $\rF$ will later be recorded in the word.

\begin{figure}[t!]
  \bigskip
  \begin{center}
    \includegraphics[scale=1]{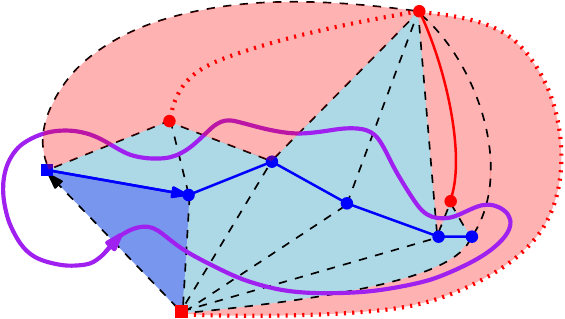}
  \end{center}
  \caption{The Mullin--Bernardi--Sheffield bijection, applied to the map in \cref{fig:Tutte}. We start from the loop crossing the root triangle in \cref{fig:Tutte}~\textbf{(c)}. Then, we enter unvisited components recursively using the following rule: flip the edge of the last traversed triangle whose companion triangle is not visited by the exploration. In the present case, we only flip primal (blue) edges to dual (red) edges; these edges are drawn in dotted line on the picture. Finally, we read the word from the space-filling exploration: every triangle with a solid edge corresponds to either $\rh,\rH$ (blue) or $\rc,\rC$ (red) depending on whether it is the first/second time that the associated quadrangle is visited. If the triangle has an edge in dotted line (i.e.\ the edge has been flipped in the aforementioned procedure), we replace the order by an $\rF$. We stress that there is a red (fictional) triangle outside of the picture that we did not represent in the drawing. The exploration shown in the figure corresponds to the word $w=\rh\rh\rh\rh\rh\rc\rc\rH\rC\rH\rH\rc\rH\rH\rF\rF$.} 
  \label{fig:sheffield_exploration}
\end{figure}

Now label the triangles in the order encountered by $\mathfrak{l}$.  
Each quadrilateral of $Q(\mathfrak m,\omega)$ consists of two companion triangles.  
The first of the pair visited by $\mathfrak{l}$ contributes a letter $\rh$ (if primal) or $\rc$ (if dual); the second contributes $\rH$ or $\rC$ respectively.  
This yields an ``intermediate'' word in $\{\rc,\rh,\rC,\rH\}$.  
Finally, for every quadrilateral whose diagonal was flipped in the merging procedure, the corresponding second letter ($\rC$ or $\rH$) is replaced by $\rF$.  The resulting word is our desired $w$.

One checks \cite[Section~4.1]{sheffield2016quantum} that companion triangles always form matching burger--order pairs under the reduction rules; in particular, this implies that every burger created from a triangle lying inside one of the original loops is consumed before the exploration leaves this loop. Therefore, we have $\overline{w}=\emptyset$.
Importantly, we note that the number of $\rF$ symbols is exactly one fewer than the number of loops in $L(\mathfrak m,\omega)$.

\paragraph{From words to decorated maps.}
We now reverse the procedure.  Let $w$ be any word of length $2k$ in $\Theta$ with $\overline{w}=\emptyset$.  Our aim is to recover a map $\mathfrak m\in\mathcal{M}_k$ along with an edge set $\omega$.

\bigskip

\noindent \emph{Step 1: reconstructing $T(\mathfrak m,\omega)$.}
First replace each $\rF$ in $w$ by $\rC$ or $\rH$ according to the type of burger with which it pairs under reduction.  Starting from an oriented root edge (with dual to primal vertices), we will build a triangulation together with a path that keeps primal edges to the left and dual edges to the right.

We read $w$ from left to right. If the current letter is $\rh$ (resp.\ $\rc$) we glue a primal (resp.\ dual) triangle to the edge just traversed, ensuring that the orientation of primal/dual sides is preserved.  
If the current letter is $\rH$ (resp.\ $\rC$), we glue the triangle in the same manner and additionally identify its primal (resp.\ dual) edge with that of its matching $\rh$ (resp.\ $\rc$) triangle; this completes a quadrangle.  
Because $\overline{w}=\emptyset$, the path eventually returns to an edge connecting the endpoints of the root, allowing us to close the figure to a rooted triangulation decorated by primal/dual types and equipped with a single space-filling loop.

From the paired triangles we obtain a rooted planar quadrangulation $Q$ by deleting primal and dual edges.

\bigskip

\noindent \emph{Step 2: recovering the fictional edges.}
In the previous substitution of $\rF$ by $\rC$ or $\rH$, we lost the information that certain quadrilaterals originally had their diagonals flipped.  
We now restore this: for every matched pair $\rc\rF$ or $\rh\rF$, we flip the corresponding diagonal of the associated quadrilateral.  
This yields a triangulation $T$.

\bigskip

\noindent \emph{Step 3: recovering the planar map.}
Since $Q$ is bipartite, its vertex set can be split into ``primal'' and ``dual'' parts.  
We define the primal part (i.e.\ $V(\mathfrak m)$) as the part containing the target vertex of the root edge of $Q$.  
Two primal vertices are joined by an edge whenever they are connected by a diagonal of a quadrilateral of $Q$; in this way we obtain the underlying planar map $\mathfrak m$.  
Those edges of $\mathfrak m$ that also appear in $T$ are declared to form $\omega$.  
A straightforward verification shows that $T=T(\mathfrak m,\omega)$, completing the reconstruction.

\medskip

In summary, the two procedures above are exact inverses, establishing a bijection between pairs $(\mathfrak m,\omega)$ with $\mathfrak m\in\mathcal{M}_k$ and words $w\in\Theta^{2k}$ satisfying $\overline{w}=\emptyset$.


\subsubsection{Random inventory accumulation}

Via the bijection recalled above, choosing at random a decorated FK planar map $(M,\Omega)$ with $k$ edges is equivalent to sampling a word $W$ of length $2k$ in the alphabet $\Theta$ with the constraint that its reduced form satisfies $\overline{W}=\emptyset$.  We now give a convenient probabilistic description of such a word.

Let $p=p(q)\in[0,1)$ be the solution of
\begin{equation}\label{eq: p}
\sqrt{q}=\frac{2p}{1-p},
\end{equation}
and, having fixed $p$, assign to every symbol $\theta\in\Theta$ a weight $\mathsf{w}(\theta)$ by
\begin{equation}\label{eq: symbol weights}
\mathsf{w}(\rc)=\mathsf{w}(\rh)=\frac14, \qquad 
\mathsf{w}(\rC)=\mathsf{w}(\rH)=\frac{1-p}{4}, \qquad
\mathsf{w}(\rF)=\frac{p}{2}.
\end{equation}
For a word $w=\theta_1\cdots\theta_n$ we also write $\mathsf{w}(w)=\prod_{i=1}^{n}\mathsf{w}(\theta_i)$. Note that $q\in(0,4)$ corresponds to $p\in(0,1/2)$.

Consider next a word $W$ of length $2k$ whose letters are drawn independently according to the distribution \eqref{eq: symbol weights}, and then condition on the event $\overline{W}=\emptyset$.  If a specific word $w$ satisfies $\overline{w}=\emptyset$, then its conditional probability is proportional to
\[
\big(\tfrac14\big)^{\#\rc+\#\rh}
\big(\tfrac{1-p}{4}\big)^{\#\rC+\#\rH}
\big(\tfrac{p}{2}\big)^{\#\rF}.
\]
Since the reduction constraint forces 
\[
\#\rc+\#\rh=\#\rC+\#\rH+\#\rF=k,
\]
the above expression simplifies to
\begin{equation}\label{eq: word law}
\mathbb{P}\!\left(W=w \,\middle|\, |W|=2k,\ \overline{W}=\emptyset\right)
\propto
\big(\tfrac{1-p}{16}\big)^{k}
\big(\tfrac{2p}{1-p}\big)^{\#\rF}.
\end{equation}

Finally, recall that the planar map $\mathfrak{m}$ corresponding to $w$ has exactly $k$ edges, and that the number of loops in $L(\mathfrak{m},\omega)$ is $\#\rF+1$.  Consequently, the conditional law of $W$ described above is precisely the law of the word associated with a configuration $(\mathfrak{m},\omega)$ drawn from \eqref{def: FK sphere}, where $p$ and $q$ are linked via \eqref{eq: p}.


\subsection{Infinite self-dual FK-decorated random planar maps}
\label{par:inf_fk_map}

\subsubsection{Definition of the model} 
\label{sss:infinite_fk}
Recall from the preceding section that a finite self-dual FK$(q)$-decorated planar map $(M,\Omega) = (M_k, \Omega_k)$ with $k$ edges and cluster-weight parameter $q$ may be encoded by a word $W$ sampled under law \eqref{eq: word law}.  The correspondence between the map and the word is provided by the Mullin--Bernardi--Sheffield bijection presented in \cref{sec:mullin-bernardi-sheffield}.  As established in \cite{chen2017basic} and \cite{sheffield2016quantum}, as $k\to\infty$ the sequence $(M_k,\Omega_k)$ converges in distribution, with respect to the local topology, to a limiting object
$(M_\infty,\Omega_\infty)$
known as the \textbf{infinite (self-dual) FK$(q)$ planar map}.  The metric governing this convergence can be taken to be
\[
d_{\mathrm{loc}}\big((\mathfrak m,\omega),(\mathfrak m',\omega')\big)
= \frac{1}{\sup\{R : B_R(\mathfrak m,\omega)=B_R(\mathfrak m',\omega')\}},
\]
where $B_R(\mathfrak m,\omega)$ denotes the set of vertices and edges of $(\mathfrak m,\omega)$ lying within graph distance $R$ of the root edge.  The ambient space is the completion of finite rooted planar maps endowed with a marked subgraph.

\medskip

A convenient description of the law of $(M_\infty,\Omega_\infty)$ begins with a \emph{bi-infinite} word
\begin{equation} \label{W_biinfinite}
W = \ldots X(-1)\,X(0)\,X(1)\ldots,
\end{equation}
whose letters $(X(i))_{i\in\mathbb Z}$ are i.i.d.\ with distribution \eqref{eq: symbol weights} (linking $p$ and $q$ through \eqref{eq: p}).  Applying Sheffield's word-to-map construction to $W$ yields the infinite FK-decorated map, in the following sense.

First, we know from \cite[Proposition~2.2]{sheffield2016quantum} that every letter in the bi-infinite sequence \eqref{W_biinfinite} has a unique \textbf{match}: each burger symbol ($\rc$ or $\rh$) is eventually consumed by an order symbol ($\rC$, $\rH$, or $\rF$), and every order corresponds to some earlier burger. The match of $X(i)$ is denoted $X(\varphi(i))$.

\medskip

Then, we can describe neighbourhoods of the root in $(M_\infty,\Omega_\infty)$, where we regard the symbol $X(0)$ as representing the root triangle.  We consider finite words $e$ of the form
\[
e = \rh\;\cdots\;\rF \qquad\text{or}\qquad e = \rc\;\cdots\;\rF,
\]
where the terminal $\rF$ matches the initial burger.  Such words are called \textbf{$\rF$-excursions} (of type $\rh$ or $\rc$, respectively).  Almost surely, $X(0)$ lies in an infinite nested family of $\rF$-excursions, and these excursions encode successively larger regions surrounding the root in the associated planar map.

Suppose that $e$ is an $\rF$-excursion of type $\rh$ containing $X(0)$:
\[
e = \rh\;\cdots\; X(0)\;\cdots\;\rF.
\]
The reduced word $\overline e$ consists only of $\rC$ symbols (if any).  By deleting these residual $\rC$'s from $e$ we obtain a word $e'$, to which we may apply the word-to-map direction of the Mullin--Bernardi--Sheffield bijection. This produces the envelope surrounding the root that corresponds to $e$.

One can iterate this over the full nested sequence of excursions containing $X(0)$ to get an exhaustion of the infinite triangulation by growing neighbourhoods of the root.  For each graph-distance radius $R$, the associated ball is described by the unique excursion whose envelope first contains all vertices within that distance. Since these maps are consistent with one another, the infinite object $(M_\infty, \Omega_\infty)$ is defined as a ``projective limit''. An alternative concrete and elegant description of $(M_\infty, \Omega_\infty)$ is given in \cite{chen2017basic}.

\subsubsection{The reduced burger walk}
\label{sec:reduced_walks}

We recall the construction of the reduced burger walk of \cite{berestycki2017critical}, which is the key exploration that we will use to derive our loop and cluster exponents in \cref{thm:exp_main}.
First, we need to introduce some terminology on hamburger-cheeseburger words. Fix a bi-infinite i.i.d. sequence $W = (X_n)_{n\in \mathbb{Z}}$ as in \eqref{W_biinfinite}.
Recall that a word $e$ is called an {$\rF$-excursion} (of $W$) if it is of the form $\rh\cdots\rF$ (type $\rh$) or $\rc\cdots\rF$ (type $\rc$), where the final $\rF$ is matched to the first letter.
We say that an $\rF$-excursion which is subword of $(X_n)_{n\leq 0}$ is \textbf{maximal} if it is not contained in any other $\rF$-excursion inside $X(-\infty, 0)$. Then we can write $X(-\infty, 0)$ in a unique way as
\begin{equation}\label{eq:dec_max_exc}
X(-\infty,0) = \cdots Y(2)Y(1)X(0),
\end{equation}
where for each $i\ge 1$, $Y(i)$ is either a single letter among $\rh$, $\rc$, $\rH$ or $\rC$, or  a maximal $\rF$-excursion. For future reference, we introduce the alphabet 
\begin{equation}\label{eq:Ah}
\mathscr{A}_{\rh} = \{ \rh, \rH, \rc\rF\} \end{equation}
(resp.\ $\mathscr{A}_{\rc}$) of all words made of $\rh$, $\rH$ and $\rF$-excursions of type $\rc$ (resp.\ $\rc$, $\rC$ and $\rF$-excursions of type $\rh$). Then \eqref{eq:dec_max_exc} is the unique way to spell $X(-\infty,0)$ in the alphabet $\mathscr{A}_{\rc} \cup \mathscr{A}_{\rh}$.

We can now define the \textbf{reduced walk} $(\tilde{h}_n,\tilde{c}_n,n\ge 0)$. We first set $\tilde{h}_0=0$ and $\tilde{c}_0=0$. Then we define $(\tilde{h}_n,\tilde{c}_n)$ recursively for $n\ge 1$ as follows: 
\begin{itemize}
    \item if $Y(n) = \rH$ (\textit{resp.} $Y(n)=\rC$), then $(\tilde{h}_n,\tilde{c}_n) = (\tilde{h}_{n-1}+1,\tilde{c}_{n-1})$ (\textit{resp.} $(\tilde{h}_n,\tilde{c}_n) = (\tilde{h}_{n-1},\tilde{c}_{n-1}+1)$);
    \item if $Y(n) = \rh$ (\textit{resp.} $Y(n)=\rc$), then $(\tilde{h}_n,\tilde{c}_n) = (\tilde{h}_{n-1}-1,\tilde{c}_{n-1})$ (\textit{resp.} $(\tilde{h}_n,\tilde{c}_n) = (\tilde{h}_{n-1},\tilde{c}_{n-1}-1)$);
    \item if $Y(n)$ is an $\rF$-excursion $E$ of type $\rh$ (\textit{resp.} $\rc$), we let
    $(\tilde{h}_n,\tilde{c}_n) = (\tilde{h}_{n-1},\tilde{c}_{n-1}+|\overline{E}|)$ (\textit{resp.} $(\tilde{h}_n,\tilde{c}_n) = (\tilde{h}_{n-1}+|\overline{E}|,\tilde{c}_{n-1})$), where $|\overline{E}|$ is the length of the reduced word. 
\end{itemize}
In words, the reduced walk corresponds to the ham and cheeseburger counts backwards from time $0$, but where each $\rF$-excursion is read all at once, accounting for a single jump (with size given by the length of the reduced $\rF$-excursion). Notice that there are times when none of the components jumps: this happens when reading an $\rF$-excursion of reduced length $0$.
Let $\tilde{\tau}^{\rh}$ and $\tilde{\tau}^{\rc}$ the hitting times of $-1$ by $\tilde{h}$ and $\tilde{c}$ respectively. We also set 
\begin{equation}\label{eq:tautilde}
\tilde{\tau} := \tilde{\tau}^{\rh} \wedge \tilde{\tau}^{\rc}.
\end{equation}

Note that the increments of the reduced walk are always one-dimensional.
Specifically, $\tilde{h}$ can only move at times corresponding to words in $\mathscr{A}_{\rh}$, while $\tilde{c}$ can only move at times corresponding to words in $\mathscr{A}_{\rc}$.
For this reason, it will often be more convenient to work with the walk $(h_n,n\ge 0)$ (resp.\ $(c_n,n\ge 0)$) skipping past times when $h$ (resp.\ $c$) stays put. Let $\tau^{\rh}$ (resp.\ $\tau^{\rc}$) the hitting times of $-1$ by $h$ (resp.\ $c$). By symmetry between ham and cheeseburgers, the two random walks {$h$ and $c$} have the same step distribution, hence $\tau^{\rh}$ and $\tau^{\rc}$ have the same distribution. Moreover, a key observation of \cite{berestycki2017critical} is that $h$ and $c$ are actually independent, by independence of the symbols in the hamburger-cheeseburger sequence, and the fact that the $\rF$-excursions are processed all at once.

The step distribution of the random walks $h$ and $c$ can be described as follows. Let $\Xi$ a random variable with law given by the reduced length of $X(\varphi(0))\cdots X(0)$ conditional on $X(0)=\rF$. Then the step distribution $\xi$ can be sampled as follows:
\begin{itemize}
	\item with probability $1/2$, set $\xi=-1$;
	\item with probability $\frac{1-p}{2}$, set $\xi=1$;
	\item with probability $\frac{p}{2}$, sample $\Xi$ as above and set $\xi = \Xi$.
\end{itemize}
In particular, we claim that the random walks $h$ and $c$ are centred. Indeed, \cite[Section 3.1]{sheffield2016quantum} implies that the variable $\Xi$ has expectation $1$: in Sheffield's notation, $\Xi$ is nothing but the law of $|X(-J,-1)|-1$ (conditional on $X(0)=\rF$, which is independent), so that $\mathbb{E}[\Xi] = \chi - 1 = 1$. Therefore 
\begin{equation} \label{eq:xi_centred}
\mathbb{E}[\xi] 
=
-\frac12 + \frac{1-p}{2} + \frac{p}{2} \mathbb{E}[\Xi]
=
0. 
\end{equation}
This highly nontrivial fact will be used crucially later on.

Given $(h,c)$, one can actually recover the lazy reduced walk $(\tilde{h}, \tilde{c})$ by adding some extra randomness. The construction is given by the following coupling.
Let $(G_i)_{i\ge 1}, (G'_i)_{i\ge 1}$ be two independent sequences of i.i.d.\ geometric random variables with parameter $1/2$, i.e., $\mathbb{P}( G_i = j) = \mathbb{P}( G'_i = j)= 2^{- j}$ for $j\ge 1$, independent of $(c,h)$. Let 
$$T_k = 
\begin{cases}    
G_1 + G'_1 + \ldots + G_{k/2} + G'_{k/2} & \text{ if $k$ even}\\
G_1 + G'_1 + \ldots + G_{\lfloor k/2 \rfloor} & \text{ if $k$ odd}
\end{cases}
$$ 
(think of $T_k$ as a renewal process where the inter-renewal times alternate between $G_i$ and $G'_i$, and $T_k$
is the time of the $k$th renewal). We also let $N_k = \sum_{i=1}^k G_i$ and $N'_k  = \sum_{i=1}^k G'_i$. 
We now define $(h^G, c^G)$ as follows. 
For $n\in[T_k, T_{k+1} - 1] $ with $k$ even, we let $c^G$ evolve as $c$ does during the interval $[N_{k/2} , N_{k/2+1 } - 1]$, while $h^G$ stays constant. That is, 
\begin{equation}\label{eq:(h,c)_geo_construc}
c^G_n = c_{n-N'_{k/2}} \text{ and }h_n^G = h_{N'_{k/2}}; \quad T_k \le n \le T_{k+1} - 1.
\end{equation}
Conversely, if $k$ is odd it is $h^G$ which will evolve during $[T_k, T_{k+1} - 1]$ while $c^G$ will remain constant.

\subsection{Bijection with the fully packed loop-$O(n)$ model on triangulations}
\label{sec:O(n)_model}

\subsubsection{Definition and connection with FK-decorated maps}
\label{sss:link_FK_O(n)}
Recall from \cref{sec:mullin-bernardi-sheffield} that the Mullin--Bernardi--Sheffield construction produces a planar rooted triangulation (the Tutte map) from any FK-decorated planar map, with every face -- including the root face -- a triangle.
In the gasket decomposition approach that we will describe later, it is useful to broaden the setting to triangulations whose root (external) face has arbitrary degree~$\ell$.
Let $\Tb_{\ell}$ denote the family of rooted planar triangulations $\tfrak$ with boundary length~$\ell$, together with a configuration $\boldsymbol{\ell}$ of disjoint simple loops drawn on the dual.
Throughout, the loop configuration is taken to be \textbf{fully packed}: every internal face of $\tfrak$ is visited by some unique loop.
Given parameters $x,n>0$, the weight assigned to $(\tfrak,\boldsymbol{\ell})\in \Tb_{\ell}$ is
\begin{equation}\label{eq: def weight triangulation}
Z(\tfrak,\boldsymbol{\ell},x,n)=n^{\# \boldsymbol{\ell}} x^{\# F(\tfrak)-1},
\end{equation}
meaning that each loop contributes a global factor $n$, and each internal (triangle) face locally contributes~$x$.
The associated \textbf{partition function} is
\begin{equation}\label{eq: def partition triangulation}
F_{\ell}=\sum_{(\tfrak,\boldsymbol{\ell})\in \Tb_{\ell}} Z(\tfrak,\boldsymbol{\ell},x,n),
\end{equation}
and whenever this sum is finite it defines the fully packed loop-$O(n)$ measure on $\Tb_{\ell}$.

\medskip
It is convenient to interpret this model as the symmetric specialisation of the twofold loop model studied in~\cite{BBG12}.
In that formulation, the loops divide the triangulation into regions coloured \emph{red} or \emph{blue}, the colouring of each triangle depending solely on the parity of the number of loops encircling that triangle. Because no loop crosses the root face, the whole boundary loop (or rather the triangles traversed by it) is monochromatic; thus the entire colouring is fixed once the boundary colour is prescribed.

In this bicoloured model, the law of the configuration $(T, L)$ is given by 
\begin{equation}
    \label{eq:lawbicolour}
    \mathbb{P} ( (T, L) = (\mathfrak{t}, \boldsymbol{\ell})) \; \propto \;  x_1^{\#\mathfrak{f}_1} x_2^{\#\mathfrak{f}_2} n^{\#\boldsymbol{\ell}},
\end{equation}
where $\mathfrak{f}_1$ and $\mathfrak{f}_2$ correspond to the (internal) triangle faces  of colour 1 (blue) and 2 (red) of $\mathfrak{t}$ respectively, and $x_1, x_2>0$. 
We will refer to the above weight as $Z(\tfrak,\boldsymbol{\ell},x_i,x_j,n):= x_1^{\#\mathfrak{f}_1} x_2^{\#\mathfrak{f}_2} n^{\#\boldsymbol{\ell}}$.

When the face weights $x_1, x_2$ are chosen symmetrically with respect to the two colours ($x_1 = x_2 = x>0$), the model collapses exactly to the fully packed loop-$O(n)$ setting above, and the partition functions for the two boundary colours coincide.

\medskip
We now recall the link with FK-decorated maps.
Implicit in the work of~\cite{BBG12} is the fact that the Tutte map $T(\mathfrak{m},\omega)$ of a self-dual FK$(q)$-weighted map $(\mathfrak{m},\omega)$ with $q=q(p)$ (see \cref{sss:def_FK}) has the law of a fully packed loop-$O(n)$ triangulation provided that the parameters are suitably chosen, which turns out to require the following relations:
\begin{equation}\label{eq:x_c_formula}
n=\frac{2p}{1-p},
\qquad
x=x_{c}(n)=\frac1{\sqrt{8(n+2)}}.
\end{equation}
One way to see this is as follows. Consider $(\tfrak,\boldsymbol{\ell})\in \Tb_{\ell}$ with boundary of fixed colour (which uniquely determines the colouring of all triangles through colour switching when crossing loops).
Attach an external vertex of the opposite colour to each of the $\ell$ boundary edges and add a loop traversing all the new triangles.
The resulting rooted triangulation therefore contains $\#F(\tfrak)+\ell-1$ faces and $\#\boldsymbol{\ell}+1$ loops.
Recalling the weight in \eqref{eq: word law}, we see that under the map-to-word direction of the Mullin--Bernardi--Sheffield bijection (\cref{sec:mullin-bernardi-sheffield}), this decorated triangulation corresponds to a word $w$ of length $2k=\#F(\tfrak)+\ell-1$ with $\overline{w}=\emptyset$ and weight
\begin{equation}\label{eq:link_FK_Z}
\Big(\frac{1-p}{16}\Big)^{\frac{\#F(\tfrak)+\ell-1}{2}}
\Big(\frac{2p}{1-p}\Big)^{\#\boldsymbol{\ell}+1}
=
n x_c^{\ell+1} x_c^{\#F(\tfrak)-1} n^{\#\boldsymbol{\ell}}
=
n x_c^{\ell+1} Z(\tfrak,\boldsymbol{\ell},x_c,n),
\end{equation}
where $n$ and $x_c$ are as in~\eqref{eq:x_c_formula}.

In the general bicoloured model, this transformation forms a one-to-one correspondence between the FK($q$) model of \eqref{eq:lawgen_s} and the loop-$O(n)$ model \eqref{eq:lawbicolour} with loop weight $n=\sqrt{q}$ and edge weights 
\[
x_1 := s\Big(\frac{p_0}{\sqrt{q}(1-p_0)} \Big)^{1/2} 
\quad \text{and} \quad x_2:= s.
\]
See \cite[Section~2.1]{BBG12}. In particular, we see that pairs $(s,p_0)$ in the region $\mathcal{R}_q$ defined in \eqref{eq:def_R_q} simply correspond to taking $x_1,x_2 \leq x_c$. In particular, our \cref{thm:F_ell_main} will entail that the corresponding partition function at $(s,p_0)$ is finite.

\subsubsection{The gasket decomposition}
\label{sss:gasket_dec}
The gasket decomposition, introduced in \cite{BBG12a} and further developed in \cite{BBG12}, provides a remarkably effective framework for analysing a broad class of loop models.
In what follows we adapt the presentation of \cite{BBG12} to the particular fully packed setting determined by \eqref{eq: def weight triangulation} and~\eqref{eq: def partition triangulation}.

Consider a configuration $(\tfrak,\boldsymbol{\ell})\in \mathbb{T}_\ell$. 
We may see $(\tfrak,\boldsymbol{\ell})$ as a bicoloured triangulation by colouring the boundary $(\tfrak,\boldsymbol{\ell})$ with colour $i=1,2$ and alternating colours when crossing the loops. We denote by $j$ the other colour.
To construct its gasket, we first look at the edges of $\tfrak$ that are reachable from the boundary without crossing any loop.
Because the loops are fully packed, these accessible edges divide the map into exactly $\#\boldsymbol{\ell}+1$ faces: the external face, with boundary length~$\ell$, and an additional face of degree~$k$ for each loop of (outer) perimeter~$k$. 
Note that all these faces have colour $i$.
The resulting object is by definition the \textbf{gasket}.
Now examine one of the internal faces of degree~$k$ created by the gasket.

\begin{figure}[b!]
  \bigskip
  \begin{center}
    \includegraphics[width=0.5\textwidth]{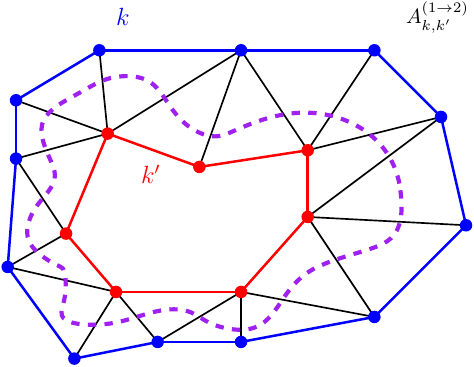}
  \end{center}
  \caption{The ring partition function $A^{(1\to 2)}_{k,k'}$ from colour 1 (blue) to colour 2 (red). It accounts for the weight of all the triangles crossed by the purple loop.} 
  \label{fig:ring_part}
\end{figure}

If we reinsert the triangles traversed by the corresponding loop, we obtain a ring made up of triangles with colour $i$ whose outer boundary length is~$k$ and triangles with colour $j$ whose inner boundary has some length~$k'$. 
See \cref{fig:ring_part}.
The ring consists of $k+k'$ triangles in total.
Moreover, the loop-decorated triangulation originally contained inside this loop --- call it $(\tfrak',\boldsymbol{\ell}')\in\mathbb{T}_{k'}$ --- is precisely what was taken out from $(\tfrak,\boldsymbol{\ell})$ to form the corresponding face of the gasket.
That said pair $(\tfrak',\boldsymbol{\ell}')$ naturally comes with the colour $j$ which is the colour of its boundary in the original map $(\tfrak,\boldsymbol{\ell})$.
In this way, $(\tfrak,\boldsymbol{\ell})$ decomposes into:
(i) the gasket;
(ii) one ring of triangles for each internal gasket face; and
(iii) a fully packed loop-decorated triangulation attached along the inner boundary of each such ring.
The contribution of such a component to the total weight 
$Z(\tfrak,\boldsymbol{\ell},x_i,x_j,n)$
in \eqref{eq: def weight triangulation} is 
$n A^{(i\to j)}_{k,k'}Z(\tfrak', \boldsymbol{\ell}', x_i,x_j, n)$,
where $k$ is the outer and $k'$ the inner boundary length of the ring, and $A^{(i\to j)}_{k,k'}$ is the partition function of rings of triangles of outer boundary with colour $i$ and size~$k$, and inner boundary of colour $j$ and size~$k'$.
This motivates the definition
{
\begin{equation}\label{eq: gasket decomposition}
g_k^{(i)} := n\sum_{k'=0}^{\infty} A^{(i\to j)}_{k,k'}  F_{k'}^{(j)},
\end{equation}
See \cref{fig:ring_part}.
Thus $g_k^{(i)}$ corresponds to the weight of any face of degree $k$ in the gasket of $(\mathfrak{t}, \boldsymbol{\ell})$ with boundary colour $i$. In other words, this gasket has the law of a \textbf{Boltzmann map} with fixed boundary length $\ell \ge 1$, and where we assign to each face of degree $m$ the weight~$g_m^{(i)}$. We denote by $\mathcal{F}_{\ell}((g_m^{(i)}){m\geq 1})$ the resulting partition function, and note that this coincides with $F_{\ell}^{(i)}$.}
Thus \eqref{eq: gasket decomposition} may be rewritten as the so-called \emph{fixed point equation}
{
\begin{equation}\label{eq: fixed point equation}
g_k^{(i)} = n\sum_{k'=0}^{\infty} A^{(i\to j)}_{k\rightarrow k'}  \mathcal{F}_{k'}((g_m^{(j)})_{m\geq 1}).
\end{equation}}
Introduce the \textbf{resolvents}
\begin{equation}\label{eq: def resolvent}
W^{(i)}(z)=\sum_{\ell=0}^{\infty} \frac{\mathcal{F}_{\ell}((g_m^{(i)})_{m\geq 1})}{z^{\ell+1}}
=\sum_{\ell=0}^{\infty}\frac{F_{\ell}^{(i)}}{z^{\ell+1}},
\end{equation}
which play a crucial role in what follows.
It is known that such a solution satisfies the \textbf{one-cut lemma} \cite{BBG12b}. More precisely, it is known that {$W^{(i)}$} is analytic on $\mathbb{C}\setminus {[\gamma^{(i)}_-,\gamma^{(i)}_+]}$ for some real interval ${[\gamma^{(i)}_-,\gamma^{(i)}_+]}$ containing zero, ${\gamma_-^{(i)}} \le 0 \le {\gamma_+^{(i)}}$, with ${\gamma_+^{(i)}} \geq |{\gamma_-^{(i)}}|$, and that it satisfies ${W^{(i)}}(z) \sim 1/z$ as $|z|\to\infty$.
Moreover, the \textbf{spectral density} ${\rho^{(i)}}$ defined by 
\begin{equation}\label{def: spectral density}
{\rho^{(i)}}(y): = - \frac{{W^{(i)}}(y+i0)-{W^{(i)}}(y-i0)}{2\pi i}, \quad y\in {[\gamma^{(i)}_-,\gamma^{(i)}_+]},
\end{equation}
is positive in ${(\gamma^{(i)}_-,\gamma^{(i)}_+)}$ and continuous  on ${[\gamma^{(i)}_-,\gamma^{(i)}_+]}$, which vanishes at both endpoints.
These properties have been established rigorously using combinatorial bijections with Motzkin paths, see \cite[Section~6]{BBG12b}.
By Cauchy integration, the spectral density determines the resolvent via the following formula:
\[
W^{(i)}(z) = \int^{\gamma^{(i)}_+}_{\gamma^{(i)}_-}
\frac{\rho^{(i)}(y)}{z-y}\mathrm{d}y, \quad  z\in\mathbb{C}\setminus[\gamma^{(i)}_-, \gamma^{(i)}_+]. 
\]
With this notation, equations (3.22)--(3.23) of \cite{BBG12} specialise to the two
{
\textbf{resolvent equations}
\begin{equation}\label{eq: general resolvent equation i}
W^{(1)}(z+i0)+W^{(1)}(z-i0)+nW^{(2)}\left(s_1(z)\right)=z,
\qquad z\in(\gamma^{(1)}_-,\gamma^{(1)}_+),
\end{equation}
and 
\begin{equation}\label{eq: general resolvent equation j}
W^{(2)}(z+i0)+W^{(2)}(z-i0)+nW^{(1)}\left(s_2(z)\right)=z,
\qquad z\in(\gamma^{(2)}_-,\gamma^{(2)}_+),
\end{equation}
where $s_i(z):= \frac{1}{x_j}(1-x_iz)$ for $i=1,2$.
Note that for the above two equations \eqref{eq: general resolvent equation i}--\eqref{eq: general resolvent equation j} to make sense, we need $s_1(z) \notin (\gamma^{(2)}_-,\gamma^{(2)}_+)$ whenever $z\in (\gamma^{(1)}_-,\gamma^{(1)}_+)$. 
See the discussion right above equation (3.19) in \cite{BBG12}. 
Moreover, $s_1$ is continuous, so $s_1(z)$ has to stay on the same side of $(\gamma^{(2)}_-,\gamma^{(2)}_+)$ for all $z\in (\gamma^{(1)}_-,\gamma^{(1)}_+)$. Since $0\in (\gamma^{(1)}_-,\gamma^{(1)}_+)$ and $s_1(0)$ is clearly positive, it must hold that $s_1(z) > \gamma^{(2)}_+$ for all $z\in (\gamma^{(1)}_-,\gamma^{(1)}_+)$. This implies $\frac{1-x_1 \gamma_+^{(1)}}{x_2} \ge \gamma^{(2)}_+$, which yields the inequality 
\begin{equation}\label{eq:bound_h1_h2}
    x_1 \gamma_+^{(1)} + x_2 \gamma_+^{(2)} \le 1.
\end{equation}
}

{In the colour-symmetric case (i.e.\ $x_1=x_2=x$), we will drop all superscripts $i$ and $j$ from the notation. Note that the two equations \eqref{eq: general resolvent equation i}--\eqref{eq: general resolvent equation j} then boil down to the single equation
\begin{equation}\label{eq: general resolvent equation}
W(z+i0)+W(z-i0)+nW\left(\frac{1}{x}-z\right)=z,
\qquad z\in(\gamma_-,\gamma_+),
\end{equation}
and that \eqref{eq:bound_h1_h2} yields 
}
$$
\gamma_+ \le 1/ (2x). 
$$
When $x = x_c$ we will see another argument based on the hamburger--cheeseburger bijection below (see Proposition \ref{prop:determ_gamma}, in particular \eqref{eq:gamma_+inequ}). One particular goal of the present paper is to solve this equation explicitly in the case when $n\in (0,2)$ and $x=x_c$ as in \eqref{eq:x_c_formula}.

The works \cite{BBG12a,BBG12b,BBG12} solve the resolvent equation \eqref{eq: general resolvent equation} when the cut $[\gamma_-, \gamma_+]$ is \emph{fixed}, giving an explicit expression in the elliptic parametrisation (in a far more general framework than fully packed triangulations). However, equation \eqref{eq: general resolvent equation} should be thought of as an equation on the triplet $(W,\gamma_-,\gamma_+)$, where the cut is part of the unknown. The aforementioned works relied on numerical evidence to support this \emph{uniqueness} ansatz (i.e., there is really a unique choice of $\gamma_-, \gamma_+$, and therefore ultimately of $W$). In the present paper, we will first derive some information about the cut (\cref{sec:self_dual_critical}), before we solve the equation. Because of our more restrictive framework of fully packed triangulations, our solution in \cref{sec:resolvents} will actually be closer to the physics paper of Gaudin and Kostov \cite{gaudin1989n}. We refer to \cite{BBD16,korzhenkova2022exploration} for a more detailed discussion.

We conclude by mentioning that this uniqueness ansatz was rigorously established by Budd and Chen \cite{budd2019peeling} in the case of rigid quadrangulations. This approach uses the symmetry relation $\gamma_- = - \gamma_+$ that holds in the case of bipartite maps and which does not appear to admit a straightforward extension to our setting.

%
%
\section{Hamburger-cheeseburger and the partition function}
\label{sec:self_dual_critical}

In this section we present a hamburger-cheeseburger argument showing that if $x = x_c$ is as in \eqref{eq:xc} then the right endpoint of the cut is given by
\[
\gamma_+ = \frac{1}{2x_c}.
\]
This extra input will allow us to solve the {(colour-symmetric)} resolvent equation \eqref{eq: general resolvent equation} in \cref{sec:resolvents}.
This will go through establishing an exact relation between the partition function $F_\ell$ and the marginal law of the filled-in cluster at the origin in the infinite FK map. This part is close to some of the arguments in \cite{da2025scaling}.

%
%
\subsection{Typical clusters as skeleton words}
\label{sec:typ_cluster}
We now show that the reduced walks of \cref{sec:reduced_walks} carry some geometric information about the FK planar maps. We will be interested in filled-in clusters since these will be easy to relate to the partition function $F_\ell$. However, we stress that similar techniques can be leveraged to access further geometric information, such as loop exponents, as done in \cite{berestycki2017critical}. Analogous results for the case $q=4$ (which is not covered here) are presented in \cite{da2025scaling}.

We first define the geometric objects that we will be working with.
Conditional on $X(0)=\rF$, we define the \textbf{bubble} or \textbf{envelope} $\mathfrak{e}(0)$ to be the (loop-decorated) submap encoded by the $\rF$-excursion $X(\varphi(0))\cdots X(0)$. We choose its root face to be the only face that is \emph{not} crossed by a loop encoded by an $\rF$ inside $e$.
Moreover, given $X(0)=\rF$, there is a \textbf{typical loop} $\mathfrak{L}(0)$ in the infinite FK planar map corresponding to that $\rF$ symbol. By the Jordan curve theorem, this loop disconnects the infinite triangulation into two connected components, and we call \textbf{typical filled-in cluster}, denoted by $\mathfrak{K}(0)$, the unique finite component. This definition makes sense so long as $X(0)=\rF$, and so from now on we place ourselves under this condition.

\begin{figure*}[b!]
    \medskip
    \centering
    \begin{subfigure}[t]{1\textwidth}
        \centering
        \includegraphics[page=1,width=0.75\textwidth]{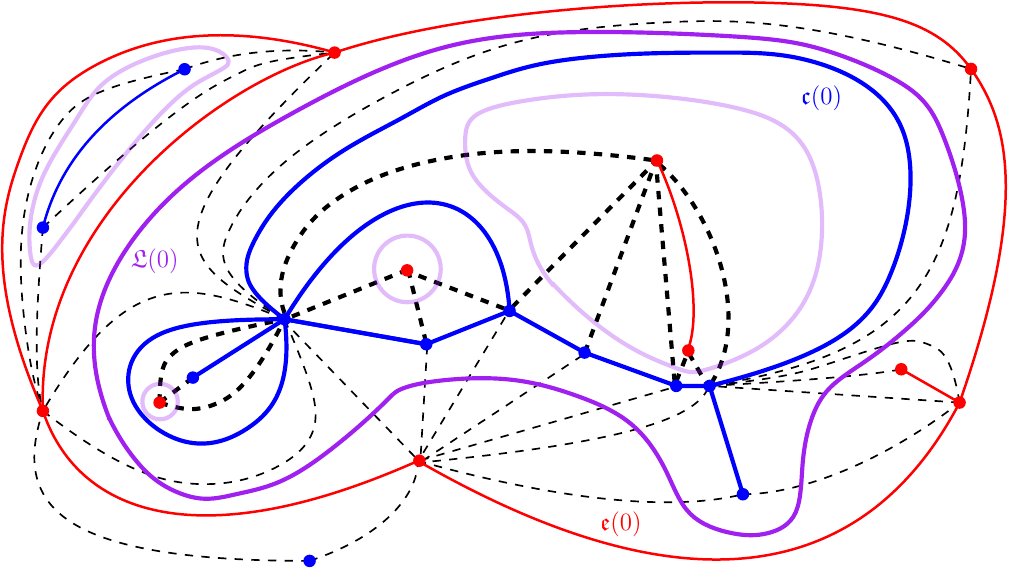}
        \caption{}
    \end{subfigure}%

    \vspace{0.5cm}
    \begin{subfigure}[t]{1\textwidth}
        \centering
        \includegraphics[page=2,width=0.75\textwidth]{loop_clusters_envelopes.pdf}
        \caption{}
    \end{subfigure}
    \caption{Loops, clusters and envelopes. The triangle at $0$ is in grey in the bottom picture and is assumed to be an $\rF$. \textbf{(a)} The corresponding typical loop $\mathfrak{L}(0)$ is shown in bold purple (other loops are shown in pale purple). The typical filled-in cluster $\mathfrak{K}(0)$  is the whole loop-decorated triangulation in the middle (in bold). The envelope $\mathfrak{e}(0)$ is the whole loop-decorated triangulation inside the red component, with its root face lying outside the drawing. 
    \textbf{(b)} The envelope is the whole loop-decorated triangulation in this bottom figure, with Sheffield's exploration in purple. We have not represented the full exploration but only the one that corresponds to the reduced walk. In particular, there is a smaller red component (to the right of the exploration) that should be visited by the space-filling exploration.
    \label{fig:loop_cluster_env}
    }
\end{figure*}

We now explain how filled-in clusters are encoded in Sheffield's bijection. The difficulty is that although each $\rF$ symbol encodes a unique loop (which corresponds to an interface between primal and dual clusters for the FK model $\omega$ on $\mathfrak{m}$), it is not the case that the submap encoded by the word (excursion) between the $\rF$ symbol and its match to the left describes a single filled-in cluster. This is because, as we circulate around a given cluster using Sheffield's exploration procedure, we will also explore a few adjacent (dual) clusters along the way, due to the rules of this exploration. For instance, in Figure \ref{fig:loop_cluster_env}, the fat purple loop $\mathfrak{L}(0)$ surrounds the primal cluster $\mathfrak{K}(0)$, but the corresponding exploration procedure (and submap encoded by the $\rF$-excursion or envelope) also explores the primal clusters nested within the adjacent dual cluster (located top left) en route, as well as the dual cluster itself.  

To account for this we introduce the notion of \emph{skeleton} of an $\rF$-excursion (or envelope) as above, which plays an important role in the analysis below.  We say that a word $w$ in the alphabet $\{\rc,\rC,\rh,\rH,\rF\}$ is a \textbf{skeleton word} of type $\rh$ if $\overline{w}= \emptyset$ and $w$ is a concatenation of words in $\mathscr{A}_{\rh} = \{ \rh, \rH, \rc\rF\}$ (recall the definition of this alphabet in \eqref{eq:Ah}). We define likewise skeleton words of type $\rc$ by swapping $\rc$ and $\rh$ in the previous definition. 
Importantly, given an $\rF$-excursion $e$ of type $\rh$ (say), one can form its skeleton decomposition $\mathsf{sk}(e)$, which is the skeleton word obtained by simply forgetting inside $e$:
\begin{itemize}
    \item 
all the sub-$\rF$-excursions of type $\rh$ that are not contained inside an $\rF$-excursion of type $\rc$, 
\item as well as all letters $\rc$ and $\rC$ that do not lie inside an $\rF$-excursion of type $\rc$. 
\end{itemize}
For instance, if
\[e=\rh \rC \rh \rC\rF \rh \rH \rc \rc \rC \rF\rF,\]
then 
\[
\mathsf{sk}(e) 
=
\rh\rh\rH\rc \rc \rC\rF\rF. 
\]
One can also see the skeleton $\mathsf{sk}(e)$ of the excursion $e$ as word in the alphabet $\mathcal{A}_\rh \cup \mathcal{A}_\rc$, by replacing every sub-excursion in $\mathsf{sk}(e)$ if type $\rh$ or $\rc$ by the letter $(\rh \rF)$ and $(\rc \rF)$, respectively. This corresponds to the ``maximal excursion decomposition'' described in \eqref{eq:dec_max_exc}; to avoid confusion we call the resulting word $\tilde{\mathsf{sk}}(e)$. For instance in the above example we would have 
$$
\tilde{\mathsf{sk}}(e) = \rh \rh \rH (\rc\rF) \rF,
$$
where the letter $(\rc \rF)$ comes from the sub-excursion $\rc \rc \rC\rF$ in $\mathsf{sk}(e)$. 

The point of introducing skeleton words is that they describe filled-in clusters, as we now state.

\begin{Prop}[Skeleton words are filled-in clusters]
	\label{prop:skeleton=clusters}
    On the event that $X(0)=\rF$, define the $\rF$-excursion word $E=X(\varphi(0))\cdots X(0)$.
	Then: 
    \begin{itemize}
        \item 
the triangles in the infinite FK map that have an edge in $\mathfrak{K}(0)$ are in one-to-one correspondence with symbols in $\mathsf{sk}(E)$. 
\item 	Moreover, triangles on the boundary of $\mathfrak{K} (0)$ (i.e., lying outside $\mathfrak{K}(0)$ and with an edge in $\mathfrak{K}(0)$) are in one-to-one correspondence with letters of $\tilde{\mathsf{sk}}(E)$ (i.e., seen as a word in $\mathscr{A}_{\rh} \cup \mathscr{A}_{\rc}$). We call $\partial \mathfrak{K}(0)$ this set of triangles.
        \end{itemize} 
	Finally, under these correspondences, the triangles are explored consecutively in Sheffield's bijection when reading $\mathsf{sk}(E)$ from left to right. 
\end{Prop}
\begin{proof}
Without loss of generality, we may assume that $X(\varphi(0))=\rh$. In that case, note that the typical filled-in cluster $\mathfrak{K}(0)$ is primal. We now write $E= X(\varphi(0))\cdots Y(2)Y(1)X(0)$ in the maximal excursion decomposition of \eqref{eq:dec_max_exc}.
We can further uniquely decompose $E$ into
\begin{equation}\label{eq:dec_skeleton}
    E = \rh R(\ell)S(\ell)\cdots R(1)S(1)R(0)\rF, 
\end{equation}
\noindent  where the $S(i)$ are either $\rh$, $\rH$ or a maximal $\rF$-excursion of type $\rc$ and the $R(i)$'s are (possibly empty) subwords in those $Y$'s which are in $\mathscr{A}_{\rc}$. Here $\ell \geq 1$ is some number, later (as a consequence of this proposition) we will see that this corresponds to the size of the boundary of $\mathfrak{K}(0)$. (Note that for $1\le i \le \ell$, we view $S(i)$ as a word on the standard alphabet $\mathcal{A}$, not a letter from the alphabet $\mathcal{A}_{\rh}$.)

Under the decomposition \eqref{eq:dec_skeleton}, we have $\mathsf{sk}(E) = S(\ell)\cdots S(1)$ by definition of $\mathsf{sk}(E)$. It remains to see that any triangle in the infinite FK map with an edge in $\mathfrak{K}(0)$ corresponds to a unique $S(i)$, for some $1\le i\le \ell$. We divide the proof of this fact into two claims.

\medskip
\noindent \textbf{Claim 1:} \emph{When $S(i) \in \{\rh,\rH\}$, the associated triangle lies outside $\mathfrak{K}(0)$ but shares an edge with $ \mathfrak{K}(0)$.}
Let $1\leq i\leq \ell$ such that $S(i) \in \{\rh,\rH\}$. Then $S(i)$ corresponds through Sheffield's bijection to a triangle $\mathfrak{T}(i)$. Since the loop configuration separating primal/dual clusters is fully packed, the triangle $\mathfrak{T}(i)$ must be crossed by some loop $\mathfrak{L}(i)$. As any loop, the loop $\mathfrak{L}(i)$ corresponds to some $\rF$ symbol. On the other hand, the triangle $\mathfrak{T}(i)$ is in the envelope $\mathfrak{e}(0)$ of $0$ since $S(i)$ appears in $E$. Therefore the envelope $\mathfrak{e}(0)$ must contain the whole ring of triangles crossed by $\mathfrak{L}(i)$. In other words, the $\rF$-excursion containing $S(i)$ corresponding to $\mathfrak{L}(i)$ is a subword of $E$. By maximality of the excursion decomposition \eqref{eq:dec_skeleton}, the only possibility is that this $\rF$-excursion is $E$, and so $\mathfrak{L}(i)=\mathfrak{L}(0)$ is the loop at $0$.
In this case, we claim that $S(i)$ corresponds to a triangle that is \emph{outside} $\mathfrak{K}(0)$ but shares an edge with it. 
To summarise, we proved that when $S(i) \in \{\rh,\rH\}$, the corresponding triangle $\mathfrak{T}(i)$ is crossed by $\mathfrak{L}(0)$. By definition of $\mathfrak{K}(0)$, we deduce that $\mathfrak{T}(i)$ lies outside $\mathfrak{K}(0)$. Moreover, this triangle is primal (because $S(i) \in \{\rh,\rH\}$ and $S(i)$ is not matched to an $\rF$), hence it must share an edge with $\mathfrak{K}(0)$ (actually $\partial \mathfrak{K}(0)$).

\medskip
\noindent \textbf{Claim 2:} \emph{When $S(i)$ is an $\rF$-excursion, all the triangles encoded by $S(i)$ lie inside $\mathfrak{K}(0)$ but one, which only shares an edge with $ \mathfrak{K}(0)$.}
Let $1\leq i\leq \ell$ such that $S(i)$ is an $\rF$-excursion (necessarily of type $\rc$). In that case, $S(i)$ corresponds to a bubble $\mathfrak{e}(i)$, which is the submap of the infinite FK map encoded by the $\rF$-excursion $S(i)$. This submap has a root face $\mathfrak{f}(i)$, which comes from the quadrangle that has its diagonal flipped by the $\rF$ symbol in Sheffield's bijection. It also has a boundary consisting of primal triangles, since $S(i)$ is an excursion of type $\rc$. Therefore, $\mathfrak{e}(i) \setminus \mathfrak{f}(i)$ has to lie inside $\mathfrak{K}(0)$, by definition of $\mathfrak{K}(0)$. Hence the triangles encoded by the word $S(i)$ all lie inside $\mathfrak{K}(0)$, except for one triangle (corresponding to the root face) which lies outside $\mathfrak{K}(0)$ and only shares an edge with $\mathfrak{K}(0)$ (in fact $\partial \mathfrak{K}(0)$).

\medskip
\noindent In any case, combining the above two claims, we get that $S(i)$ only encodes triangles that share an edge with $\mathfrak{K}(0)$. Conversely, we claim that a triangle that shares an edge with $\mathfrak{K}(0)$ must be encoded by a symbol appearing in one of the $S(i)$, $1\le i\le \ell$. Again there are two cases:
\begin{itemize}
	\item If the triangle lies inside $\mathfrak{K}(0)$, Sheffield's exploration would have to first enter $\mathfrak{K}(0)$, then later encode that triangle, and finally exit $\mathfrak{K}(0)$. This means that the symbol corresponding to that triangle appears in an $\rF$-excursion of type $\rc$.
	\item If the triangle lies outside $\mathfrak{K}(0)$ but shares an edge with it, then in particular it has to be a primal triangle crossed by $\mathfrak{L}(0)$. Such a triangle cannot be encoded by $\rc$ or $\rC$ (or else it would be dual), nor can it be correspond to a symbol inside an $\rF$-excursion of type $\rh$ (since it is crossed by $\mathfrak{L}(0)$ and primal). 
\end{itemize} 
This proves the converse, and we can therefore conclude that triangles with an edge in $\mathfrak{e}(0)$ are in one-to-one correspondence with symbols in $\mathsf{sk}(E)=S(\ell)\cdots S(1)$.
The second claim of \cref{prop:skeleton=clusters} also follows from the previous dichotomy: from Claims 1 and 2 we see that triangles outside $\mathfrak{K}(0)$ sharing an edge with $\partial \mathfrak{K}(0)$ correspond either to $S(i)\in\{\rh,\rH\}$ or to one specific symbol in $S(i)$ when it is an $\rF$-excursion.
Finally, the last claim of \cref{prop:skeleton=clusters} is straightforward since we did not change the ordering of triangles.
\end{proof}

The previous translation of filled-in clusters into skeleton words has the following consequence. Recall from \cref{sec:reduced_walks} the hitting times $\tau^{\rh}$ and $\tau^{\rc}$.
\begin{Cor}[Reduced walk expression for $|\partial \mathfrak{K}(0)|$]
	\label{cor:red_walk_c(0)}
	On the event $X(0)=\rF$ and $X(\varphi(0))=\rh$ (resp.\ $X(\varphi(0))=\rc$), we have $|\partial \mathfrak{K}(0)| = \tau^{\rh}-1$ (resp.\ $|\partial \mathfrak{K}(0)| = \tau^{\rc}-1$).
\end{Cor}
\begin{proof}
	Without loss of generality, assume $X(\varphi(0))=\rh$. By \cref{prop:skeleton=clusters}, the number of triangles outside $\mathfrak{K}(0)$ with an edge in $\partial \mathfrak{K}(0)$ is the number $\ell$ of letters in the alphabet $\mathscr{A}_{\rh}$ in the skeleton decomposition $\tilde{\mathsf{sk}}(E)=\tilde S(\ell)\cdots \tilde S(1)$ of $E=X(\varphi(0))\cdots X(0)$. In other words, $|\partial \mathfrak{K}(0)| = \ell$. Furthermore, these letters correspond precisely to the times where the reduced hamburger walk $h$ changes. Moreover, reading $\mathsf{sk}(E)=S(\ell)\cdots S(1)$ backwards from $S(1)$, we see that $h$ stays nonnegative until time $\ell$, since $E$ is an $\rF$-excursion of type $\rh$. Then, we have $h(\ell+1)=-1$ since the next increment of $h$ after time $\ell$ corresponds to finding the match $X(\varphi(0))=\rh$ of $X(0)=\rF$. Therefore, we conclude that $\tau^{\rh}=\ell+1=|\partial \mathfrak{K}(0)|+1$.
\end{proof}

There is also an analogous representation for the boundary length of the typical loop $\mathfrak{L}(0)$. Recall that the perimeter $|\mathfrak{L}(0)|$ of $\mathfrak{L}(0)$ is defined as the number of triangles it crosses (in particular, we always have $|\mathfrak{L}(0)| \geq |\partial \mathfrak{K}(0)|$.

\begin{Prop}[Reduced walk expression for $|\mathfrak{L}(0)|$]
\label{prop:red_walk_L(0)}
    On the event that $X(0) = \rF$, we have $|\mathfrak{L}(0)| = \tilde{\tau}$ (where we recall that $\tilde \tau$ was defined in \eqref{eq:tautilde}).
\end{Prop}

\begin{proof}
    The proof is similar to that of \cref{prop:skeleton=clusters}.
    We again write $E= X(\varphi(0))\cdots Y(1)X(0)$ in the maximal excursion decomposition of \eqref{eq:dec_max_exc}. 
    We first claim that, if $X(0)=\rF$, then $\varphi(0) = \tilde{\tau}$. Indeed, if $X(0) = \rF$, the match of $0$ is the first (negative) time where one has a net surplus of any type of burger production ($\rh$ or $\rc$), which is $\tilde{\tau}$. 

    We then extend \textbf{Claim 1} in the proof of \cref{prop:skeleton=clusters} to the following statement: when $Y(i)\in \{\rh,\rc,\rH,\rC \}$ for some $i\in \{ 1,\ldots, \tilde{\tau}\}$, the associated triangle is crossed by the loop $\mathfrak{L}(0)$. Indeed, it is crossed by some loop since the configuration is fully packed. If it were crossed by another loop than $\mathfrak{L}(0)$, then the symbol $Y(i)$ would lie inside another sub-$\rF$-excursion, which would contradict the maximality of the excursion decomposition.

    We also extend \textbf{Claim 2} in the proof of \cref{prop:skeleton=clusters} to the following statement: when $Y(i)$ is an $\rF$-excursion for some $i\in \{ 1,\ldots, \tilde{\tau}\}$, only one of the triangles encoded by $Y(i)$ is crossed by $\mathfrak{L}(0)$. Indeed, the word $Y(i)$ is an envelope contained (strictly) inside $\mathfrak{e}(0)$. Such an envelope only has one face that is crossed by the loop $\mathfrak{L}(0)$, which is its root face (corresponding to the triangle that is flipped when entering the envelope). 

    These two claims together prove that $|\mathfrak{L}(0)| = \varphi(0) = \tilde{\tau}$.
\end{proof}

%
%
\subsection{Typical cluster marginals}
From the description of the typical filled-in cluster as a skeleton word (\cref{prop:skeleton=clusters}), we can express the marginal law of $\mathfrak{K}(0)$ in terms of the loop-$O(n)$ weights \eqref{eq: def weight triangulation}. By symmetry, we may assume that the boundary of $\mathfrak{K}(0)$ is primal (blue), or equivalently $X(\varphi(0)) = \rh$.

\begin{Prop}[Typical filled-in cluster marginals] \label{prop:outer_bdry_prob}
Let $(\tfrak,\boldsymbol{\ell})\in \Tb_\ell$ a rooted loop-decorated triangulation with boundary length $\ell \ge 1$. Let $N_{\ell+1}$ be the sum of $(\ell+1)$ i.i.d.\ geometric random variables with parameter $1/2$, {independent of $\tau^{\rc}$}. Then there is a normalising constant $C>0$ (that does not depend on $\ell$) such that
\[\mathbb{P}(\mathfrak{K}(0) = \tfrak \mid X(0) = \rF, X(\varphi(0)) = \rh)
=
C Z(\tfrak,\boldsymbol{\ell}, x_c, n) (2x_c)^{\ell+1}\mathbb{P}(\tau^{\rc} > N_{\ell+1}).\] 
\end{Prop}

\begin{proof}
In this proof, we use the symbol $\propto$ to indicate proportionality between two sides: we stress that, although the sides might depend on $\ell$, the proportionality constant will never depend on $\ell$.
Fix a rooted loop-decorated triangulation $(\tfrak,\boldsymbol{\ell})\in \Tb_\ell$.
Through Sheffield's bijection, $\tfrak$ can be encoded as a hamburger-cheeseburger word $w$ with $\overline{w}=\emptyset$. Likewise, conditioned on $X(0) = \rF$ and $X(\varphi(0)) = \rh$, the bubble $\mathfrak{e}(0)$ at $0$ is encoded by the $\rF$-excursion $E:=X(\varphi(0))\cdots X(0)$. Then \cref{prop:skeleton=clusters} entails
\begin{equation} \label{eq:sk_weight_sum}
\mathbb{P}(\mathsf{sk}(E) = w \mid X(0) = \rF, X(\varphi(0)) = \rh)
= \sum_{e\in\mathcal{S}(w)} \mathbb{P}(E=e \mid X(0) = \rF, X(\varphi(0)) = \rh),
\end{equation}
where $\mathcal{S}(w)$ is the set of $\rF$-excursions $e$ of type $\rh$ such that $\mathsf{sk}(e)=w$. 
Decomposing $e\in \mathcal{S}(w)$ into maximal excursion decomposition as in \cref{sec:reduced_walks}, we can write
\begin{equation} \label{eq:e_r(i)_s(i)}
e = \rh r(\ell)s(\ell)\cdots r(1)s(1)r(0)\rF, 
\end{equation}
where the $s(i)$ are in $\mathscr{A}_{\rh}$, and the $r(i)$'s  possibly empty words in $\mathscr{A}_{\rc}$. Likewise, write
\[
E 
= X(\varphi(0)) R(\ell)S(\ell)\cdots R(1)S(1)R(0) X(0)
= \rh R(\ell)S(\ell)\cdots R(1)S(1)R(0)\rF.
\]
As in \eqref{eq:dec_skeleton} we note from the decomposition \eqref{eq:e_r(i)_s(i)} that $\mathsf{sk}(e) = s(\ell)\cdots s(1)$ by definition of the skeleton decomposition. Therefore the $s(i)$, $1\le i\le \ell$, of a word $e\in\mathcal{S}(w)$ are fixed by the condition that $\mathsf{sk}(e) = w$. 

Now what is the set of admissible $r(0),\ldots,r(\ell) \in \mathscr{A}_{\rc}$ such that $e \in \mathcal{S}(w)$? By definition, each $r(i)$, $0\le i\le \ell$, is a word in the alphabet $\mathscr{A}_{\rc}$. The only constraint for $r(i)$, $0\le i\le \ell$, to be admissible is that the first $\rh$ symbol in \eqref{eq:e_r(i)_s(i)} be matched to the final $\rF$ symbol. In turn, this means that $\tau^{\rc} > \sum_{i=0}^{\ell} |r(i)|$, where $|r(i)|$ denotes the length of $r(i)$ seen as an element of $\mathscr{A}_{\rc}$. Moreover, the hitting time $\tau^{\rc}$ does not depend on the skeleton $\mathsf{sk}(e)$, which only encodes the increments of $h$. As a consequence, the probability in \eqref{eq:sk_weight_sum} is proportional to
\begin{multline} \label{eq:sk_weight_r(i)}
\mathbb{P}(\mathsf{sk}(E) = w \mid X(0) = \rF, X(\varphi(0)) = \rh) \\
\propto \quad \mathsf{w}(s(1))\cdots \mathsf{w}(s(\ell)) \sum_{r(0),\ldots,r(\ell)\text{ words in } \mathscr{A}_{\rc}} \mathsf{w}(r(0))\cdots \mathsf{w}(r(\ell)) \mathds{1}_{\{\tau^{\rc}>\sum_{i=0}^{\ell} |r(i)|\}},
\end{multline}
where we recall that $\mathsf{w}$ denotes the hamburger-cheeseburger weights in \eqref{eq: symbol weights}. In the above expression, we have taken the weight of the empty word to be $1$ (recall that any $r(i)$ could be empty). 
By the correspondence between hamburger-cheeseburger and loop-$O(n)$ weights in \eqref{eq:link_FK_Z}, we have $\mathsf{w}(s(1))\cdots \mathsf{w}(s(\ell)) \; \propto \; x_c^{\ell+1} Z(\tfrak,\boldsymbol{\ell},x_c, n)$, whence
\begin{multline} \label{eq:sk(E)_prop_sum}
\mathbb{P}(\mathsf{sk}(E) = w \mid X(0) = \rF, X(\varphi(0)) = \rh) \\
\propto \quad x_c^{\ell+1} Z(\tfrak,\boldsymbol{\ell},x_c, n) \sum_{r(0),\ldots,r(\ell) \text{ words in } \mathscr{A}_{\rc}} \mathsf{w}(r(0))\cdots \mathsf{w}(r(\ell)) \mathds{1}_{\{\tau^{\rc}>\sum_{i=0}^{\ell} |r(i)|\}}.
\end{multline}

Recall our coupling of the lazy walk $(\tilde{h},\tilde{c})$ with $(h,c)$ at the end of \cref{sec:reduced_walks}. In particular, we pointed out that the amounts of time between two steps of $\tilde{h}$ are distributed as i.i.d.\ geometric random variables with parameter $1/2$. Specifically, if we decompose the word $X(-\infty,0)$ into
\[
X(-\infty,0) = \cdots R(2)S(2)R(1)S(1)R(0)X(0),
\]
where, as in \eqref{eq:e_r(i)_s(i)}, the $S(i)$ are in $\mathscr{A}_{\rh}$ and the $R(i)$ are (possibly empty) words in $\mathscr{A}_{\rc}$, then the length $|R(i)|$ of $R(i)$ in the alphabet $\mathscr{A}_{\rc}$ is a geometric random variable.
In addition, the probability that $(R(0),\ldots, R(\ell))$ equals some fixed sequence $(r(0),\ldots, r(\ell))$ of words in $\mathscr{A}_{\rc}$ is proportional to the weight $\mathsf{w}(r(0))\cdots \mathsf{w}(r(\ell))$.
Recalling our notation $N_{\ell+1}$ from the coupling, the previous discussion translates into
\[
	\mathbb{P}(\tau^{\rc}> N_{\ell+1})
	=
	\mathbb{P}\bigg(\tau^{\rc}>\sum_{i=0}^{\ell} |R(i)|\bigg)
	=
	\frac{\sum_{r(0),\ldots,r(\ell) \text{ words in } \mathscr{A}_{\rc}} \mathsf{w}(r(0))\cdots \mathsf{w}(r(\ell)) \mathds{1}_{\{\tau^{\rc}>\sum_{i=0}^{\ell} |r(i)|\}}}{\sum_{r(0),\ldots,r(\ell) \text{ words in } \mathscr{A}_{\rc}} \mathsf{w}(r(0))\cdots \mathsf{w}(r(\ell))}.
\]
We can thus simplify the sum in \eqref{eq:sk(E)_prop_sum} as 
\[
\mathbb{P}(\tau^{\rc}> N_{\ell+1}) \cdot \sum_{r(0),\ldots,r(\ell)\in \mathscr{A}_{\rc}} \mathsf{w}(r(0))\cdots \mathsf{w}(r(\ell)).
\]

Finally, it remains to analyse the contribution of the above sum. The weight $\mathsf{w}(r(i))$ of $r(i)$ is obviously the same for each $i$, so that we can focus on $r(0)$.
Since $r(0)$ is a word in the alphabet $\mathscr{A}_{\rc}$, let us write it as $r(0)=y(k)\cdots y(1)$ with $y(1), \ldots, y(k) \in \mathscr{A}_{\rc}$. Then 
\[
\sum_{r(0)\in \mathscr{A}_{\rc}} \mathsf{w}(r(0)) 
= \sum_{k\ge 0} \sum_{y(1),\ldots,y(k)\in\mathscr{A}_{\rc}} \mathsf{w}(y(k))\cdots \mathsf{w}(y(1))
= \sum_{k\ge 0} \bigg(\sum_{y\in\mathscr{A}_{\rc}} \mathsf{w}(y)\bigg)^k,
\] 
where the weight for $k=0$ in the second expression is interpreted as $1$ (recall our discussion following \eqref{eq:sk_weight_r(i)}). Furthermore, by symmetry between ham and cheeseburgers, the total weight $\sum_{y\in \mathscr{A}_{\rc}} \mathsf{w}(y)$ of each block is $1/2$. Summarising, we get that $\sum_{r(0)\in \mathscr{A}_{\rc}} \mathsf{w}(r(0)) = \sum_{k\ge 0} 2^{-k} = 2$, and hence
\[
\sum_{r(0),\ldots,r(\ell)\in \mathscr{A}_{\rc}} \mathsf{w}(r(0))\cdots \mathsf{w}(r(\ell))
=
2^{\ell+1}.
\]
Going back to \eqref{eq:sk(E)_prop_sum}, we conclude that
\[
	\mathbb{P}(\mathfrak{K}(0) = \tfrak \mid X(0) = \rF, X(\varphi(0)) = \rh)
	\quad \propto \quad
	Z(\tfrak,\boldsymbol{\ell}, x_c, n) (2x_c)^{\ell+1}\mathbb{P}(\tau^{\rc} > N_{\ell+1}),
\]
which is our claim.
\end{proof}

%
%
\subsection{Dictionary and proof of ansatz}
We now have the tools we need in order to express our correspondence between loop-$O(n)$ statistics and FK planar maps. The correspondence takes the form of an exact identity between the partition function $F_\ell$ of the loop-$O(n)$ model with boundary size $\ell$, and hitting times of the reduced (burger) walk.
This will essentially follow from summing \cref{prop:outer_bdry_prob} over all maps in $\Tb_{\ell}$ and the connection that we made in \cref{cor:red_walk_c(0)} between the reduced walk and the typical filled-in cluster. 

We then use this identity to deduce the value of the right endpoint of the cut $\gamma_+$, thereby proving the ansatz underlying the analysis in \cite{BBG12} and discussed in the introduction and preliminaries.

\begin{Prop}[From hitting times to the partition function]\label{prop:prob_tau}
	There exists a normalising constant $C>0$ such that, for all $\ell\geq 0$,
	\[
	\mathbb{P}(\tau^{\rh} = \ell+1) = C (2x_c)^{\ell+1}F_{\ell}. 
	\]
\end{Prop}
	
	\begin{proof}
	Summing over all $(\tfrak,\boldsymbol{\ell}) \in \Tb_{\ell}$ in \cref{prop:outer_bdry_prob}, we get
	\[
	\mathbb{P}(|\partial\mathfrak{K}(0)| = \ell \mid X(0) = \rF, X(\varphi(0)) = \rh)
	= C (2x_c)^{\ell+1}F_{\ell} \mathbb{P}(\tau^{\rc} > N_{\ell+1}).
	\]
	\noindent On the event that $X(0) = \rF$, the fact that $X(\varphi(0)) = \rh$ means that $\tilde{\tau}^{\rh} < \tilde{\tau}^{\rc}$. Furthermore, by \cref{cor:red_walk_c(0)}, on the event that $X(0) = \rF$ and $X(\varphi(0)) = \rh$, the perimeter $|\partial\mathfrak{K}(0)|$ is equal to $\tau^{\rh}-1$. Hence the previous display leads to
	\begin{equation} \label{eq:tau^h=ell_prop}
	\mathbb{P}(\tilde{\tau}^{\rh}< \tilde{\tau}^{\rc}, \tau^{\rh} = \ell + 1) 
	= C (2x_c)^{\ell+1}F_{\ell}\cdot\mathbb{P}(\tau^{\rc} > N_{\ell+1}).    
	\end{equation}
	
	We now make use again of our coupling $(h^G,c^G)$ in \eqref{eq:(h,c)_geo_construc}. We claim that, on the event $\{\tau^{\rh} = \ell+1\}$, the event $\{\tilde{\tau}^{\rh}< \tilde{\tau}^{\rc}\}$ is nothing but $\{\tau^{\rc}>N_{\ell + 1}\}$. Indeed, to recover $\tilde{\tau}^{\rh}$ from $\tau^{\rh}$, we only need to glue back in the intervals of time when $\tilde{h}$ stays put while $\tilde{c}$ moves, whose lengths are given by the independent geometric random variables $G_i$, $i\geq 0$. Therefore, if $\tau^{\rh}=\ell+1$, then $\tilde{c}$ hits $-1$ after $\tilde{h}$ if, and only if, the non-lazy walk $c$ hits $-1$ after $N_{\ell+1}$.
	
	By independence in the construction of the coupling, the probability factors out as
	\begin{equation}\label{eq:tau^h=ell_construct}
	\mathbb{P}(\tilde{\tau}^{\rh}< \tilde{\tau}^{\rc}, \tau^{\rh} = \ell+1) = \mathbb{P}(\tau^{\rh} = \ell+1)\mathbb{P}(\tau^{\rc}>N_{\ell + 1}). 
	\end{equation}
	\noindent We conclude from \eqref{eq:tau^h=ell_prop} and \eqref{eq:tau^h=ell_construct} that
	\[
	\mathbb{P}(\tau^{\rh} = \ell+1) = C (2x_c)^{\ell+1}F_\ell. 
	\]
    This concludes the proof.
	\end{proof}

	Recall from \cref{sss:gasket_dec} the definition of the resolvent $W$ and the existence of the cut. We deduce the value of the right end $\gamma_+$ of the cut from the above \cref{prop:prob_tau}.  

\begin{Prop}[Determination of $\gamma_+$]
	\label{prop:determ_gamma}
	We have $\gamma_+ = (2x_c)^{-1}$.
\end{Prop}

	\begin{proof}
        By \cref{prop:prob_tau}, we can write 
		\begin{equation} \label{eq:F_ell_tauh}
		F_\ell = C^{-1} (2x_c)^{-\ell-1} \mathbb{P}(\tau^{\rh}=\ell+1).
		\end{equation}
		Let $\varphi(z) := W(1/z) = \sum_{\ell\geq 0} F_{\ell} z^{\ell+1}$. By \cref{sss:gasket_dec}, 
        The function $\varphi$ is a power series, let $R_\ph$ be its radius of convergence. Since the coefficients of $\ph$ are positive we deduce from the Vivanti--Pringsheim theorem that it cannot be continued analytically to any neighbourhood of $R_\ph$. Furthermore, the arguments of \cite{BBG12b} show that it can be continued analytically to $\mathbb{C} \setminus ( ( -\infty, 1/\gamma_-]) \cup [1/\gamma_+, \infty))$ by definition of $\gamma_-$ and $\gamma_+$. 
        This implies $ R_\ph = 1/\gamma_+$ necessarily (and $|\gamma_-| \le \gamma_+$ as already announced). 

        Now let use Proposition \ref{prop:prob_tau} to first show that $\gamma_+\le 1/ (2x_c)$. To do this we bound crudely $\mathbb{P}( \tau^\rh = \ell +1) \le 1$ and deduce that $F_\ell \le C^{-1} (2x_c)^{ - \ell - 1}$. Hence the radius of convergence $R_\ph = 1/ \gamma_+$ is at least $2x_c$. In other words, 
    \begin{equation}\label{eq:gamma_+inequ} \gamma_+ \le \frac1{2x_c}.
    \end{equation}
          It remains to prove the reverse inequality. 

		 Recall from \eqref{eq:xi_centred} that the hamburger reduced walk $h$ is centred. This implies that $\mathbb{E}[\tau^{\rh}]=\infty$: if not, Wald's identity $\mathbb{E}[h_{\tau^{\rh}}] = \mathbb{E}[\tau^{\rh}]\mathbb{E}[\xi]$ would provide a contradiction, since $\mathbb{E}[h_{\tau^{\rh}}] = -1$ and $\mathbb{E}[\xi]=0$. All we need is that $\tau^\rh$ has a tail which does not decay exponentially fast: for then, by \eqref{eq:F_ell_tauh}, this implies that the radius of convergence of $\varphi$ is at most  $2x_c$. 
         With the previous paragraph we conclude that $\gamma_+ = (2x_c)^{-1}$.
	\end{proof}

	\begin{Rk}
		Using \cite[Lemma A.2]{berestycki2017critical} (see also the end of the proof of Theorem 1.1 there), one can see that
		\[
			\mathbb{P}(\tau^{\rh} = \ell+1) \sim \ell^{-2+\theta +o(1)}
			\quad 
			\text{as } \ell\to\infty,
		\]
		with $\theta = \frac{1}{\pi}\arccos(n/2)$. Together with \eqref{eq:F_ell_tauh}, this would yield 
		\[
		F_\ell \sim C \frac{(2x_c)^{-\ell}}{\ell^{2-\theta+o(1)}} \quad \text{as } \ell\to \infty,
		\]
        for some constant $C>0$.
		Actually, we will deduce more precise asymptotics and even an exact expression for $F_\ell$, using a completely different method (and without using \cite{berestycki2017critical}): namely, we will solve the resolvent equation \eqref{eq: general resolvent equation} explicitly, relying on \cref{prop:determ_gamma}. This is the purpose of the next section. Note that translating back through \cref{prop:prob_tau}, this will also imply more precise asymptotics than in \cite{berestycki2017critical}, and even an exact expression of, say, the marginals of $\tau^{\rh}$. This will be carried out in \cref{sec:loop-clusters}.
	\end{Rk}

%
%
\section{Solution to the resolvent equation in the self-dual case}
\label{sec:resolvents}
Our aim is to solve the resolvent equation in \eqref{eq: general resolvent equation} at $x=x_c(n)$, i.e.\ 
\begin{equation} \label{eq: sd resolvent}
W(z-i0)+W(z+i0) + nW\left(\frac{1}{x_c}-z\right) = z, \quad z\in(\gamma_-,\gamma_+). 
\end{equation}
Recall that we are after solutions $(W,\gamma_-,\gamma_+)$ of \eqref{eq: sd resolvent} such that $W$ has the \emph{one-cut} $(\gamma_-,\gamma_+)$ and satisfies $W(z) \sim 1/z$ as $|z|\to\infty$, see \cref{sss:gasket_dec}. Recall also from \cref{sss:gasket_dec} the definition of the spectral density $\rho$. The one-cut assumption, together with the asymptotics $W(z) = O(1/z)$ as $|z|\to\infty$, implies that:
\begin{itemize}
\item[(i)] Away from the cut,
\[
W(z) = \int_{\gamma_-}^{\gamma_+} \frac{\rho(y)}{z-y} \mathrm{d}y, \quad z\notin (\gamma_-,\gamma_+);
\]
\item[(ii)] On the cut, by the Sokhotski--Plemelj theorem,
\[
W(z\pm i0) = \dashint_{\gamma_-}^{\gamma_+} \frac{\rho(y)}{z-y} \mathrm{d}y \mp i\pi \rho(z), \quad z\in (\gamma_-,\gamma_+),
\]
where  $\displaystyle\dashint_A^B$ denotes principal-value integration. 
\end{itemize}
Thus we may rewrite \eqref{eq: sd resolvent} as 
\begin{equation} \label{eq: sd resolvent rho}
\dashint_{\gamma_-}^{\gamma_+} \rho(y)\left(\frac{1}{z-y} - \frac{n/2}{z+y-1/x_c}\right) \mathrm{d}y = \frac{z}{2}, \quad z\in(\gamma_-,\gamma_+). \tag{E}
\end{equation}
Moreover, the assumption that $W(z) \sim 1/z$ as $|z|\to\infty$ translates into the normalisation
\begin{equation} \label{eq: sd normalisation}
\int_{\gamma_-}^{\gamma_+} \rho(y) \mathrm{d}y = 1. \tag{N}
\end{equation}
Importantly, remark that the set of equations \eqref{eq: sd resolvent rho}--\eqref{eq: sd normalisation} is the same as that of Gaudin and Kostov \cite{gaudin1989n} (see equations (10) and (11) there, with $g=0$). We now solve this equation for \emph{fixed} $\gamma_-,\gamma_+$  making mathematical sense of the ideas of \cite{gaudin1989n}. Recall from \cref{prop:determ_gamma} that we identified~$x_c=\frac{1}{2\gamma_+}$.

\paragraph{Step 1: Reduction to a convolution equation.}
First, we plug into \eqref{eq: sd resolvent rho} the change of variables 
\[
u = \frac12 \log\left(\frac{\gamma_+-\gamma_-}{\gamma_+-z}\right),
\]
\[
v = \frac12 \log\left(\frac{\gamma_+-\gamma_-}{\gamma_+-y}\right).
\]
Then
\[
z = \gamma_+ -(\gamma_+-\gamma_-)\mathrm{e}^{-2u}, \quad y = \gamma_+ -(\gamma_+-\gamma_-)\mathrm{e}^{-2v},
\]
so that, setting $\rhotilde(v) = \rho(\gamma_+ -(\gamma_+-\gamma_-)\mathrm{e}^{-2v})$, since $x_c = 1/ (2\gamma_+)$ by the ansatz (Proposition \ref{prop:determ_gamma}),
the integral becomes
\begin{align*}
&2 \dashint_{0}^{\infty} (\gamma_+-\gamma_-)\mathrm{e}^{-2v} \rhotilde(v) \left(\frac{1}{(\gamma_+-\gamma_-)(\mathrm{e}^{-2v}-\mathrm{e}^{-2u})}+\frac{n/2}{(\gamma_+-\gamma_-)(\mathrm{e}^{-2v}+\mathrm{e}^{-2u})}\right) \mathrm{d}v \\
&=
2 \dashint_{0}^{\infty} \mathrm{e}^{u-v} \rhotilde(v) \left(\frac{1}{\mathrm{e}^{u-v}-\mathrm{e}^{v-u}}+\frac{n/2}{\mathrm{e}^{u-v}+\mathrm{e}^{v-u}}\right) \mathrm{d}v \\
&=
 \dashint_{0}^{\infty} \mathrm{e}^{u-v} \rhotilde(v) \left(\frac{1}{\sinh(u-v)}+\frac{n/2}{\cosh(u-v)}\right) \mathrm{d}v.
\end{align*}
In the end, \eqref{eq: sd resolvent rho} becomes
\[
\dashint_{0}^{\infty} \mathrm{e}^{-v} \rhotilde(v) \left(\frac{1}{\sinh(u-v)}+\frac{n/2}{\cosh(u-v)}\right) \mathrm{d}v
=
\frac12 \mathrm{e}^{-u} (\gamma_+ - (\gamma_+-\gamma_-)\mathrm{e}^{-2u}), 
\quad u\in (0,\infty).
\]
Defining
\begin{align}
k(z) &:= \frac{1}{\sinh(z)} +\frac{n/2}{\cosh(z)}, \label{eq:expr_k(z)}\\
f(u) &:= \frac12 \mathrm{e}^{-u}(\gamma_+-\gamma_-)(\gamma_+ - (\gamma_+-\gamma_-)\mathrm{e}^{-2u}), \label{eq:expr_f(u)} \\
r(v) &:= (\gamma_+-\gamma_-)\mathrm{e}^{-v} \rhotilde(v), \notag
\end{align}
we have ended up with 
\begin{equation} \label{eq: sd resolvent 1}
\dashint_0^\infty r(v) k(u-v) \mathrm{d}v = f(u), \quad u\in(0,\infty). 
\end{equation}
On the other hand, \eqref{eq: sd normalisation} becomes 
\begin{equation} \label{eq: sd normalisation 1}
\int_0^\infty \mathrm{e}^{-v} r(v)  \mathrm{d}v = \frac12, \quad u\in(0,\infty). 
\end{equation}

Since we will take Fourier transforms later on, it will be convenient to extend the previous functions to the whole real line, i.e.\ we set
\[
r_+(v) := 
\begin{cases} 
r(v) & \text{if } v>0, \\
0 & \text{otherwise}
\end{cases}
\quad 
\text{and}
\quad
f_+(u) := 
\begin{cases} 
f(u) & \text{if } u>0, \\
0 & \text{otherwise}.
\end{cases}
\]
We may rewrite \eqref{eq: sd resolvent 1} as
\begin{equation} \label{eq: sd resolvent extended}
\dashint_{-\infty}^\infty r_+(v) k(u-v) \mathrm{d}v = f_+(u)+f_-(u), \quad u\in\bb R,
\end{equation}
where $f_-$ is some function defined by the above identity, with $f_-(u) = 0$ for $u>0$.
We collect the following regularity properties of the function $r_+$ (i.e., $r$) and $f_-$.

\begin{Lem}\label{lem:r(v)/v_int}
The function $v \in\mathbb{R} \mapsto r_+(v)$ is bounded and integrable. Furthermore, $f_-$ is a continuous function on $\mathbb{R}$ satisfying the following bounds at 0 and $-\infty$:
\begin{align*}
    f_-(u)  = O \Big( \log \Big(\tfrac{1}{|u|}\Big) \Big),  & \quad u \to 0^-\\
    f_-(u) = O(\mathrm{e}^u), & \quad   u \to -\infty.
\end{align*}
\end{Lem}

\begin{proof}
  The first point follows from the definition of $r$ and known properties of the spectral density. Indeed, recall that by definition, $r (v) = (\gamma_+-\gamma_-)\mathrm{e}^{-v} \rhotilde(v)$. Since $\rhotilde$ is obtained by reparametrising the spectral density $\rho$, which is known to be continuous on $(\gamma_-, \gamma_+)$ and tend to 0 and $\gamma^-$ and $\gamma^+$, we first see $\rho$ is continuous and bounded, hence $\rhotilde$ is also continuous and bounded (tending to 0 and $0$ and $\infty$). In turn, $r(v)$ is continuous and bounded (tending to zero at 0 and $\infty$). Also $r_+(v) \le ( \gamma_+ - \gamma_-) \| \rhotilde\|_\infty \mathrm{e}^{-v}$ which proves integrability. This proves the first point. 

  Now let us turn to the behaviour of $f_-$. When $u<0$ we have, by definition of $f_-$, 
  \begin{align*}
f_- (u) & = 
\dashint_{0}^\infty r_+(v) k(u-v) \mathrm{d}v\\ 
& = \int_{0}^\infty r_+ (v) k(u-v) \mathrm{d} v,
  \end{align*}
  where in the first line, we used that $f_+(u) = 0$ and that $r_+(v) = 0$ for $v <0$, and in the second line, we used the fact that the principal value is an actual integral: indeed $k(z) $ only has a singularity of the form $k(z) \sim 1/z$ near $z = 0$, 
  which never arises as $u<0$ and $k$ is evaluated at $u - v$ with $v>0$. 

  Consider first the behaviour as $u \to -\infty$. Then we note that, since $r_+$ is bounded, and since $k(z) \le 4(n+1/2) \mathrm{e}^z$ for $z<0$ with $|z| $ sufficiently large, (letting $C>0$ be a constant whose value can change from line to line), 
  \[
f_-(u) \le C \int_0^\infty k(u-v) \mathrm{d} v 
 \le  4(n+1/2)C  \int_0^\infty \mathrm{e}^{ u - v} \mathrm{d} v \le C \mathrm{e}^u, 
  \]
  as desired. The behaviour at $ u = 0^-$ follows similarly after noting that $k(z) \le C / |z|$ for $z<0$ and $|z|$ sufficiently small. 
\end{proof}

\paragraph{Step 2: Taking Fourier transforms.}
The following result deals with the Fourier transforms of all the quantities in Step 1. For a function $g:\bb R\to\bb R$, we denote its Fourier transform by
\[
G(\omega) := \int_{\bb R} g(y) \mathrm{e}^{i\omega y} \mathrm{d}y, 
\quad \omega \in \bb C,
\]
whenever it is defined, possibly in the sense of principal-value integration.

\begin{Lem} 
	\label{lem:fourier_WH}
The respective Fourier transforms $(R_+, K, F_+, F_-)$ of $(r_+, k, f_+, f_-)$ satisfy the following properties:
\begin{itemize}
	\item $R_+$ is well-defined and holomorphic on the half-plane $\cal H_+ = \{z\in\bb C, \Im(z) > 0\}$.
	\item $K$ is well-defined as a Cauchy principal value and holomorphic on the strip $\{z\in\bb C, -1<\Im(z) < 1\}$. Moreover, $K$ has the explicit expression
\begin{equation} \label{eq: expression K}
K(\omega)
:=
\dashint_\bb{R}  k(y) \mathrm{e}^{i\omega y} \mathrm{d}y
=
\frac{\pi}{\cosh(\pi\omega/2)} \Big(\frac{n}{2}+ i\sinh(\pi\omega/2)\Big).
\end{equation}
	\item $F_+$ is well-defined and holomorphic on $\{z\in\bb C, \Im(z) > -1\}$. Moreover, $F_+$ has the explicit expression
\begin{equation}\label{eq: expr_F(omega)}
F_+(\omega)
:=
\int_\bb{R}  f_+(y) \mathrm{e}^{i\omega y} \mathrm{d}y
=
\frac{c_0}{1-i\omega} + \frac{c_1}{3-i\omega},  
\end{equation}
with
    $c_0:= \frac12 \gamma_+(\gamma_+-\gamma_-) \text{ and } c_1 = -\frac12 (\gamma_+-\gamma_-)^2$.
	\item $F_-$ is well-defined and holomorphic on $\cal H_- := \{z\in\bb C, \Im(z) < 1\}$.
\end{itemize}
\end{Lem}

\begin{proof}
The first and third items are a consequence of the following facts, respectively:
\begin{itemize}
\item Because $r_+$ vanishes on $\bb R_-$, we have, whenever the right-hand side makes sense, 
\[
R_+(\omega) := \int_0^{\infty} r_+(y) \mathrm{e}^{i\omega y}.
\]
Since $r_+$ is continuous and bounded, the above integral makes sense whenever $\omega \in \mathcal{H}_+$. 
\item $f_+$ has the expression \eqref{eq:expr_f(u)} on $\bb R_+^*$ and vanishes on $\bb R_-$.
\end{itemize}

For the second item, we use the expression of $k$ in \eqref{eq:expr_k(z)}, from which we see that $k$ is continuous except at $0$ where it has a $1/z$ singularity, and $k(z) = O(\mathrm{e}^{-|z|})$ as $|z|\to\infty$. Thus, $K$ is indeed well-defined as a Cauchy principal value on $\{z\in\bb C, -1<\Im(z) < 1\}$. Moreover, it is known that $K$ has the said explicit expression, see e.g.\ formulas F60-F61 in \cite[Chapter 13]{raade1995mathematics} (and so, in particular, it is analytic). 

Finally, we turn to $F_-$. Note that identity \eqref{eq: sd resolvent extended} implies that for all $u\in\bb R_-^*$,
\begin{equation} \label{eq:f-_expr}
f_-(u) = \dashint_{-\infty}^\infty r_+(v) k(u-v) \mathrm{d}v = \int_0^\infty r_+(v) k(u-v) \mathrm{d}v.
\end{equation}
We stress that the above integral is actually well-defined as a regular integral since $k$ only has a singularity at $0$. This shows that $f_-$ is continuous on $\bb R_-^*$. It also has a limit as $u\to 0^-$ since $v \mapsto r_+(v)/v$ is integrable (\cref{lem:r(v)/v_int}).
Finally, using that $k(z)= O(\mathrm{e}^{-|z|})$ as $|z|\to\infty$ and the fact that $r_+$ is integrable, we see that $f_-(u) = O(\mathrm{e}^{u})$ as $u\to -\infty$,
and so $F_-$ is well-defined and holomorphic on $\cal H_- = \{z\in\bb C, \Im(z) < 1\}$.
\end{proof}

\noindent In particular, \cref{lem:fourier_WH} implies that $(R_+, K, F_+, F_-)$ are all well-defined and holomorphic on the strip $\cal S := \{z\in\bb C, 0<\Im(z) <1\}$.
In addition, we claim that the Fourier transform of 
\[
u\in \bb R \mapsto \dashint_{-\infty}^\infty r_+(v) k(u-v) \mathrm{d}v
\]
is the product $R_+ K$. This can be seen by adding and subtracting $\frac{1}{u-v}$ to the kernel $k(u-v)$: the part $k(u-v)-\frac{1}{u-v}$ is continuous so that we can apply the general product rule for regular convolutions, while the part $\frac{1}{u-v}$ is dealt with using the product rule for Hilbert transforms.
As a consequence, upon taking Fourier transforms, our equation \eqref{eq: sd resolvent extended} then implies that for $\omega\in\cal S$,
\begin{equation} \label{eq: sd resolvent fourier}
R_+(\omega) K(\omega) = F_+(\omega) + F_-(\omega), \tag{$\hat{\text{E}}$}
\end{equation}
whereas the normalising condition \eqref{eq: sd normalisation 1} becomes
\begin{equation} \label{eq: sd normalisation fourier}
R_+(i) = \frac12. \tag{$\hat{\text{N}}$}
\end{equation}
To summarise, equation \eqref{eq: sd resolvent fourier} holds for $\omega$ in the strip $\cal S := \{z\in\bb C, 0<\Im(z) <1\}$. We see $K$ and $F_+$ as given data holomorphic on $\cal S$ (with their expressions given as in \cref{lem:fourier_WH}), and we are looking for unknowns $R_+$ and $F_-$ solving \eqref{eq: sd resolvent fourier}, which are holomorphic functions on $\cal H_+ :=\{z\in\bb C, \Im(z)>0\}$ and $\cal H_- :=\{z\in\bb C, \Im(z)<1\}$ respectively. Since $\cal S = \cal H_- \cap \cal H_+$, this corresponds to a Wiener-Hopf problem, which may be solved using standard techniques that we describe in the next step.

\paragraph{Step 3: Solution through the Wiener-Hopf factorisation.}
The key starting point is to write the \emph{Wiener-Hopf factorisation} of the kernel $K$. More precisely, we can write 
\[
K(\omega) = 2\pi^2 \frac{K_-(\omega)}{K_+(\omega)}
\quad \text{ for } \omega\in\cal S, 
\]
where the factors $K_{\pm}$ are holomorphic on $\cal H_\pm$. The next lemma makes this decomposition explicit.

\begin{Lem}[Wiener-Hopf factorisation of $K$]
	\label{lem:WH_K}
	The Wiener-Hopf factorisation of $K$ is given by
	\begin{equation} \label{eq: WH factorisation K}
	K(\omega) = 2\pi^2 \frac{K_-(\omega)}{K_+(\omega)}, \quad \omega\in\cal S,
	\end{equation}
	with\footnote{Here we stress that \cite[Equation (32)]{gaudin1989n} has a typo: the term $\frac{1-i\omega}{4}$ in the Gamma function should be replaced by $\frac{1-i\omega}{2}$.}
	\begin{equation} \label{eq: expressions K_+ and K_-}
	K_+(\omega)
	=
	\frac{\Gamma\left( \frac{3+2\theta-i\omega}{4} \right) \Gamma\left( \frac{3-2\theta-i\omega}{4} \right)}{2^{i\omega/2} \Gamma\left(\frac{1-i\omega}{2}\right)}
	\quad \text{and} \quad
	K_-(\omega)
	=
	\frac{2^{-i\omega/2} \Gamma\left(\frac{1+i\omega}{2}\right)}{\Gamma\left( \frac{1+2\theta+i\omega}{4} \right) \Gamma\left( \frac{1-2\theta+i\omega}{4} \right)},
	\end{equation}
	where $\theta = \frac{1}{\pi}\arccos(n/2)$. The functions $K_{\pm}$ are holomorphic on $\cal H_\pm$ respectively.
\end{Lem}
\begin{proof}
	Observe that since $n\in (0,2)$, we have $2\theta\in (0,1)$. The holomorphicity of $K_{\pm}$ on $\cal H_\pm$ is then a consequence of the analytic properties of the Gamma function (note that $\Gamma\left( \frac{1-2\theta+i\omega}{4} \right)$ has a pole at $\omega=i(1-2\theta)$ but the singularity is removable for $K_-$). The proof of the factorisation \eqref{eq: WH factorisation K} results from plain calculations using Euler's reflection formula for $\Gamma$.
\end{proof}

\noindent With \cref{lem:WH_K}, equation \eqref{eq: sd resolvent fourier} implies 
\[
2\pi^2 \frac{R_+(\omega)}{K_+(\omega)} = \frac{F_+(\omega)}{K_-(\omega)} + \frac{F_-(\omega)}{K_-(\omega)} \quad \text{for } \omega\in\cal S \text{ such that } \Gamma\left( \frac{1-2\theta+i\omega}{4} \right) \neq \infty.
\]
We now multiply the latter display to remove the pole singularities of $F_+$: for $\omega\in\cal S$ such that $\Gamma\left( \frac{1-2\theta+i\omega}{4} \right) \neq \infty$,
\[
2\pi^2 \frac{R_+(\omega)}{K_+(\omega)} (\omega+i)(\omega+3i)= \left(\frac{F_+(\omega)}{K_-(\omega)} + \frac{F_-(\omega)}{K_-(\omega)}\right)(\omega+i)(\omega+3i).
\]
Now remark that the left-hand side is actually holomorphic on $\cal H_+$, whereas the right-hand side is holomorphic on $\{z\in\bb C, \Im(z)<1-2\theta\}$. We stress that, since $0<n<2$, we have $1-2\theta\in (0, 1)$. Therefore, both sides can be extended to an entire function $S$ in the complex plane, whence
\begin{equation} \label{eq: WH final 1}
R_+(\omega)
=
\frac{1}{2\pi^2} K_+(\omega) \frac{S(\omega)}{(\omega+i)(\omega+3i)}
\quad \text{for } \omega\in \cal H_+,
\end{equation}
and 
\begin{equation} \label{eq: WH final 2}
F_+(\omega)+F_-(\omega) 
=
K_-(\omega) \frac{S(\omega)}{(\omega+i)(\omega+3i)} \quad \text{for } \omega \text{ such that } \Im(\omega)<1-2\theta.
\end{equation}
We now analyse the behaviours at infinity in \eqref{eq: WH final 1} and \eqref{eq: WH final 2}. 
{Since $r_+$ is integrable by Lemma \ref{lem:r(v)/v_int}, we have that $R_+(\omega) = O(1)$ as $|\omega| \to \infty$ with $\omega\in \mathcal{H}_+$. Likewise, writing $ \omega = x + {i} y$ with $y < 1$, we have by Lemma \ref{lem:r(v)/v_int},
\[
|F_-(\omega)| = \Big|\int_{-\infty}^0 f_-(u) \mathrm{e}^{ i \omega u} \mathrm{d} u  \Big|\le \int_{-\infty}^0 O(\mathrm{e}^u) \mathrm{e}^{-yu}\mathrm{d} u.
\]
Thus we deduce that  $F_- (\omega) = O(1)$ as $|\omega| \to \infty$ with $\Im(\omega)<1-2\theta$.
Combining this with the exact expressions \eqref{eq: expressions K_+ and K_-}, we see that
$S(\omega) = O(\omega)$ as $|\omega|\to\infty$. By Liouville's theorem, $S$ is a degree $1$ polynomial. 
We can therefore write it as $S(\omega) = A\omega + B$ for some $A, B \in \mathbb{C}$. We will in fact show it is a constant: indeed, we can read the value of $A$ by taking $\omega\to +\infty$ along the real line in \eqref{eq: WH final 2}. By the Riemann--Lebesgue lemma, we have $F_-(\omega) \to 0$ and this forces $A=0$ (as $F_+ ( \omega) = O ( 1/|\omega|) \to 0$ as $\omega \to \infty$). We conclude that $S$ is actually a constant,
the value of which is determined by \eqref{eq: sd normalisation fourier}.} We get (using the well known identities for the $\Gamma$ function that $\Gamma (1 + x) = x \Gamma( x)$ for at least $x >0$ and $\Gamma(x) \Gamma(1-x) = \pi / \sin(\pi x)$ whenever $x \notin \mathbb{Z}$),
\begin{equation} \label{eq:S=cst}
S = -\frac{8\pi\sqrt{2}}{\theta} \sin\left( \frac{\pi\theta}{2}\right).
\end{equation}

\paragraph{Conclusion.}
We finally arrived at the following explicit expression for $R_+$: by \eqref{eq: WH final 1} and \eqref{eq:S=cst},
\begin{equation}\label{eq: R_+ explicit}
	R_+(\omega) = 
    - \frac{4\sqrt{2}}{\pi \theta} \sin\left( \frac{\pi\theta}{2}\right) \frac{K_+(\omega)}{(\omega+i)(\omega+3i)}
    =
    - \frac{4\sqrt{2}}{\pi \theta} \sin\left( \frac{\pi\theta}{2}\right) 
    \frac{\Gamma\left( \frac{3+2\theta-i\omega}{4} \right) \Gamma\left( \frac{3-2\theta-i\omega}{4} \right)}{2^{i\omega/2} \Gamma\left(\frac{1-i\omega}{2}\right) (\omega+i)(\omega+3i)}.
\end{equation}
On the other hand, one may now use \eqref{eq: WH final 2} to get the value of $(\gamma_-,\gamma_+)$. Indeed, $F_-$ is holomorphic on $\cal H_-$, while $F_+$ has poles at $-i$ and $-3i$. Thus, we can read off the coefficients $c_0$ and $c_1$ of \eqref{eq: expr_F(omega)} from the right-hand side of \eqref{eq: WH final 2}. We arrive at $c_0 = \frac{4}{\theta} \cos\left(\frac{\pi\theta}{2}\right)\sin\left(\frac{\pi\theta}{2}\right)$ and $c_1 = -\frac{4}{\theta} \sin^2\left(\frac{\pi \theta}{2}\right)$. Recalling the expression of $c_0, c_1$ in \cref{lem:fourier_WH}, this in turn implies that 
\begin{equation}\label{eq:gamma+}
\gamma_+ = 2^{3/2} \cos\left( \frac{\pi\theta}{2}\right) \quad \text{and} \quad \gamma_+-\gamma_- = \frac{2^{3/2}}{\theta} \sin\left(\frac{\pi\theta}{2}\right).
\end{equation}
One then recovers $r_+$ as the inverse Fourier transform of the above \eqref{eq: R_+ explicit}: 
\begin{equation}\label{eq:r+explicit}
	r_+(v) 
	=
	 \frac{\gamma_+-\gamma_-}{\sqrt{2}\pi\theta} \mathrm{e}^{-3v} \left( (\mathrm{e}^{2v}+\sqrt{\mathrm{e}^{4v}-1})^{\theta} - (\mathrm{e}^{2v}-\sqrt{\mathrm{e}^{4v}-1})^{\theta} \right).
\end{equation}
We will prove this identity in Appendix \ref{sec:appendix} (note that this differs from the analogous identity in \cite[(38)]{gaudin1989n}, though their final answer appears to be correct). 
Undoing the change of variables, we conclude that for $y\in(\gamma_-,\gamma_+)$,
\begin{multline}
	\label{eq:rho_final_expr}
	\rho(y) = \\
	\frac{2^{-\theta - 1/2}}{\pi\theta(\gamma_+-\gamma_-)} (\gamma_+-y)^{1-\theta} \bigg( \Big( \sqrt{2\gamma_+ - \gamma_- -y} + \sqrt{y-\gamma_-}\Big)^{2\theta} - \Big( \sqrt{2\gamma_+ - \gamma_- -y} - \sqrt{y-\gamma_-}\Big)^{2\theta}\bigg).
\end{multline}
Note that, as expected, 
	\[
	\rho(y)
	\sim
	c (\gamma_+ - y)^{1-\theta}, \quad \text{as } y\to (\gamma_+)^-.
	\]

We can now complete the proof of Theorem \ref{thm:F_ell_main}.

\begin{proof}[Proof of \cref{thm:F_ell_main}]
	By definition, we have for all $z\notin (\gamma_-,\gamma_+)$,
	\[
	W(z) = \int_{\gamma_-}^{\gamma_+} \frac{\rho(y)}{z-y} \mathrm{d}y.
	\]
	For $|z|$ large enough, we may expand the integrand as a geometric series to get
	\[
	W(z) = \sum_{\ell\ge 0} \frac{1}{z^{\ell+1}} \int_{\gamma_-}^{\gamma_+} \rho(y) y^{\ell} \mathrm{d}y.
	\]
	We conclude, by identifying the coefficients with \eqref{eq: def resolvent}, that 
	\[
	F_\ell = \int_{\gamma_-}^{\gamma_+} \rho(y) y^{\ell} \mathrm{d}y, 
	\quad \ell\geq 0.
	\]
	The asymptotic part of the theorem (see \cref{cor:asymp_F_ell_main}) comes from plugging the expression of $\rho$ derived in \eqref{eq:rho_final_expr} and estimating the integral. The main idea for the latter estimation is to perform the change of variables $y(u)= \ell \gamma_+ u$ and then apply Lebesgue's dominated convergence theorem.
\end{proof}

\section{Exact loop and cluster exponents at the self-dual point}
\label{sec:loop-clusters}
We now use the asymptotic behaviour of $F_\ell$ to deduce the tail behaviour of typical loops and filled-in clusters in the FK$(q)$-weighted planar map model. It is probably possible to make the constants explicit, although we did not pursue this.

\paragraph{Typical cluster exponent.}
We start with the typical filled-in cluster $\mathfrak{K}(0)$, which has been introduced in \cref{sec:typ_cluster}. The next result gives the precise tail behaviour of the outer boundary length $|\partial\mathfrak{K}(0)|$ of the typical filled-in cluster.

\begin{Thm}[Typical cluster exponent]\label{thm:cluster_exp}
There is a positive constant $C>0$ such that we have the asymptotic
\[\mathbb{P}(|\partial\mathfrak{K}(0)| = \ell\mid X(0) = \rF) 
\sim 
\frac{C}{\ell^{3-2\theta}}
\quad \text{as } \ell \to \infty. \]
\end{Thm}

\begin{proof}
    The idea is to use \cref{cor:red_walk_c(0)} to rephrase the event in terms of $\tau^\rh$ and then use the expression in \cref{prop:prob_tau} together with the asymptotic behaviour of $F_{\ell}$. 
    
    Let $\ell>0$. Without loss of generality, we can assume by symmetry that the match of $0$ is a hamburger:
    \[
    \mathbb{P}(|\partial\mathfrak{K}(0)| = \ell\mid X(0) = \rF) = 
    2\mathbb{P}(|\partial\mathfrak{K}(0)|= \ell, X(\varphi(0)) = \rh \mid X(0) = \rF).
    \]
    Using \cref{cor:red_walk_c(0)} and then \eqref{eq:tau^h=ell_construct}, we deduce that
    \begin{equation} \label{eq:cluster_tau_N_ell}
    \mathbb{P}(|\partial\mathfrak{K}(0)| = \ell\mid X(0) = \rF) 
    = 
    2\mathbb{P}(\tilde{\tau}^\rh < \tilde{\tau}^\rc, \tau^{\rh} = \ell+1) 
    = 
    2\mathbb{P}(\tau^{\rh} = \ell+1)\mathbb{P}(\tau^{\rc} > N_{\ell + 1}), 
    \end{equation}
    where we recall that $N_{\ell+1}$ is the sum of $(\ell+1)$ i.i.d.\ geometric random variables with parameter $1/2$, independent of $\tau^{\rc}$.
    For the first term, we put together \cref{prop:prob_tau} and \cref{thm:F_ell_main}. It gives that 
    \begin{equation}
    \label{eq:asympt_tauh}
    \mathbb{P}(\tau^{\rh} = \ell+1)
    \sim
    \frac{c}{\ell^{2-\theta}}
    \quad \text{as } \ell\to\infty.
    \end{equation}
    For the second term, we first note that $\frac{N_{\ell+1}}{\ell} \to 1$ almost surely as $\ell\to\infty$, by the law of large numbers. Therefore, by independence between $\tau^\rc$ and $N_{\ell+1}$ (and using again \cref{prop:prob_tau} and \cref{thm:F_ell_main}), the second term gives
    \[
    \mathbb{P}(\tau^{\rc} > N_{\ell + 1})
    \sim
    c \mathbb{E}\bigg[\frac{1}{N_{\ell+1}^{1-\theta}}\bigg]
    \sim
    \frac{c}{\ell^{1-\theta}}
    \quad \text{as } \ell\to\infty.
    \]
    Plugging these two asymptotics back into \eqref{eq:cluster_tau_N_ell} and tracing the constants, we conclude the proof of \cref{thm:cluster_exp}.
\end{proof}

\paragraph{Typical loop exponent.}
Our second result provides the tail behaviour for the length $|\mathfrak{L}(0)|$ of the typical loop $\mathfrak{L}(0)$ (see \cref{sec:typ_cluster}), defined as the number of triangles it crosses. 

\begin{Prop}[Loop length exponent]\label{prop: loop_expo}
We have the asymptotic: there is a constant $C>0$ such that:
\[
\mathbb{P}(|\mathfrak{L}(0)| = \ell \mid X(0)=\rF) \sim  
\frac{C}{\ell^{3-2\theta}} \quad \text{as } \ell \to \infty. 
\]
\end{Prop}

\begin{proof}
By \cref{prop:red_walk_L(0)}, we know that $|\mathfrak{L}(0)|= \tilde{\tau}$ on the event that $X(0) = \rF$. Moreover, by symmetry between burgers, 
\begin{equation} \label{eq:L(0)_tau^h}
\mathbb{P}(|\mathfrak{L}(0)| = \ell \mid X(0)=\rF) 
= 
2\mathbb{P}(|\mathfrak{L}(0)| = \ell, X(\varphi(0)) = \rh \mid X(0)=\rF) 
= 
2\mathbb{P}(\tilde{\tau}^{\rh} = \ell, \tilde{\tau}^\rh<\tilde{\tau}^\rc). 
\end{equation}
We now use the coupling construction \eqref{eq:(h,c)_geo_construc}. Notice that when $\tau^\rh = k$, the event $\{\tilde{\tau}^\rh = \ell, \tilde{\tau}^\rh < \tilde{\tau}^\rc\}$ is equal to $\{N_k = \ell-k, \tau^\rc > N_k\}$. 
Indeed, recall that one recovers $\tilde{\tau}^{\rh}$ from $\tau^{\rh}$ by adding  in the intervals of time when $\tilde{h}$ stays put (while $\tilde{c}$ may move), whose lengths are given by the independent geometric random variables $G_i$, $i\geq 0$ (a similar idea already appeared in the proof of \cref{prop:prob_tau}).
Summing up over all possibilities for $\tau^\rh$ and using independence, we thus get
\begin{align} 
\mathbb{P}(\tilde{\tau}^{\rh} = \ell, {\tilde{\tau}^{\rh}<\tilde{\tau}^{\rc}}) 
&= 
\sum^{\ell}_{k=1}\mathbb{P}(\tau^\rh = k, \tilde{\tau}^{\rh} = \ell, {\tilde{\tau}^{\rh}<\tilde{\tau}^{\rc}}) \notag \\
&= 
\sum^{\ell}_{k=1}\mathbb{P}(\tau^\rh = k, N_k = \ell-k,\tau^\rc > N_k) \notag \\
&= 
\sum^{\ell}_{k=1}\mathbb{P}(\tau^\rc > \ell-k)\mathbb{P}(N_k = \ell-k)\mathbb{P}(\tau^\rh = k). \label{eq:Lfrak_0_sum}
\end{align}
Because of the middle term, the main contribution of this sum comes from values of $k$ that are close to $\ell/2$.

To make this more precise, let $\varepsilon \in (0,1/2)$. By a Chernoff bound, there exists $\delta = \delta(\varepsilon) > 0$, such that when $|k - \ell/2| > \varepsilon\ell$
\begin{equation*}
\mathbb{P}(N_k=\ell-k) \leq \mathrm{e}^{-\delta \ell}. 
\end{equation*}
Therefore 
\begin{equation}\label{eq: proof loop expo 1}
\sum_{|k - \ell/2|> \varepsilon\ell}\mathbb{P}(\tau^\rc > \ell-k)\mathbb{P}(N_k = \ell-k)\mathbb{P}(\tau^\rh = k) \leq \sum_{|k - \ell/2|> \varepsilon\ell}\mathbb{P}(N_k = \ell-k) \leq \ell \mathrm{e}^{-C\ell}, 
\end{equation}
which means that this part of the sum will not contribute to the asymptotics.

We now focus on the terms $|k - \ell/2| \leq \varepsilon\ell$ in the sum. From the asymptotics \eqref{eq:asympt_tauh}, we obtain for $\ell$ large enough and all $k$ such that $|k - \ell/2|\leq \varepsilon\ell$ that 
\[
\frac{c-\varepsilon}{(1/2+\varepsilon)^{1-\theta}} \frac{\ell^{\theta-1}}{1-\theta}
\leq 
(c-\varepsilon)\frac{(\ell - k)^{\theta-1}}{1-\theta} 
\leq 
\mathbb{P}(\tau^{\rc} > \ell-k) 
\leq 
(c+\varepsilon) \frac{(\ell - k)^{\theta-1}}{1-\theta}
\leq 
\frac{c+\varepsilon}{(1/2-\varepsilon)^{1-\theta}}\frac{\ell^{\theta-1}}{1-\theta}, 
\]
and likewise
\[
\frac{c-\varepsilon}{(1/2+\varepsilon)^{2-\theta}} \ell^{\theta-2}
\le
\mathbb{P}(\tau^{\rh} = k)
\le 
\frac{c+\varepsilon}{(1/2-\varepsilon)^{2-\theta}}\ell^{\theta-2}.
\]
Therefore, we can bound the sum \eqref{eq:Lfrak_0_sum} from above and below:
\begin{multline}
\frac{(c-\varepsilon)^2}{(1/2+\varepsilon)^{3-2\theta}} \frac{\ell^{2\theta-3}}{1-\theta}\sum_{|k - \ell/2|\leq \varepsilon\ell}\mathbb{P}(N_k = \ell-k) \\
\leq 
\sum_{|k - \ell/2|\leq \varepsilon\ell}\mathbb{P}(\tau^\rc > \ell-k)\mathbb{P}(N_k = \ell - k)\mathbb{P}(\tau^\rh = k) \\
\leq
\frac{(c+\varepsilon)^2}{(1/2-\varepsilon)^{3-2\theta}}\frac{\ell^{2\theta-3}}{1-\theta} \sum_{|k - \ell/2|\leq \varepsilon\ell}\mathbb{P}(N_k = \ell-k). \label{eq: proof loop expo 2}
\end{multline}

By \eqref{eq: proof loop expo 1} and \eqref{eq: proof loop expo 2}, we conclude that the whole sum in \eqref{eq:Lfrak_0_sum} satisfies, for $\ell$ large enough,
\[
\frac{(c-\varepsilon)^2}{(1/2+\varepsilon)^{3-2\theta}} \frac{\ell^{2\theta-3}}{1-\theta} 
\leq
\sum^{\ell}_{k=1}\mathbb{P}(\tau^\rc > \ell-k)\mathbb{P}(N_k = \ell-k)\mathbb{P}(\tau^\rh = k) 
\leq
\frac{(c+\varepsilon)^2}{(1/2-\varepsilon)^{3-2\theta}}\frac{\ell^{2\theta-3}}{1-\theta}.
\]
Since this holds for any $\varepsilon$, recalling \eqref{eq:L(0)_tau^h}, we end up with 
\[
\mathbb{P}(|\mathfrak{L}(0)| = \ell \mid X(0)=\rF) 
\sim
\frac{2^{2\theta-3} c^2}{1-\theta}\frac{1}{\ell^{3-2\theta}}
\quad \text{as } \ell\to\infty.
\]
This completes the proof of \cref{prop: loop_expo}.
\end{proof}

\section{Exponential decay away from the self-dual point}
\label{sec:_offsd}
In this section, we mainly focus on the FK($q$) planar map model away from self-duality and prove \cref{thm:main_offsd}.
We take the parameters $(s,p_0)$ in the region $\mathcal{S}_q$ defined in \eqref{eq:def_R_q}.
Recall that, through the correspondence described in \cref{sss:link_FK_O(n)}, the FK($q$) model \eqref{eq:lawgen_s} with percolation parameter $p_0 \ne p_{\mathrm{sd}}(q)$ is equivalent to a bicoloured fully packed loop-$O(n)$ model on triangulations with non matching face weights $x_1$ and $x_2$.
Choosing $(s,p_0)\in\mathcal{S}_q$ then corresponds to $x_1,x_2 \leq x_c$ and $(x_1,x_2)\neq (x_c,x_c)$, as noted at the end of \cref{sss:link_FK_O(n)}.
Henceforth, we thus consider the bicoloured loop-$O(n)$ model with weights $x_1 \leq x_2$ satisfying these relations (note that we can always assume $x_1 \leq x_2$ by swapping the roles of colours $1$ and $2$). 
From now on, the weights $x_1$ and $x_2$ are fixed as above.

Let us first recall some notation.
The notation $\mathfrak{K}$ refers to the cluster of the root triangle (that corresponds to the cluster of the root edge in the original FK model), which may be of either colour. Recall that $K_{\ell}^{(i)}$, $i=1,2$, is the partition function of Fortuin--Kasteleyn maps such that $\mathfrak{K}$ has colour $i$ and perimeter $\ell$. Alternatively --- by the above correspondence --- this is the partition function of (bicoloured) loop-decorated maps in the loop-$O(n)$ model on triangulations. 
The purpose of this section is to provide an argument to show that under the above condition on the weights $x_1$ and $x_2$, the partition function $K_{\ell}^{(i)}$ is exponentially decaying as $\ell \to \infty$, for both $i=1,2$, which is the gist of \cref{thm:main_offsd}.
For $i=1,2$ we let $F_\ell^{(i)}$ and $\gamma_+^{(i)}$ the partition function and right endpoint of the cut corresponding to the colour $i$ in the bicoloured loop-$O(n)$ model with weights $x_1$ and $x_2$, see our notation in \cref{sss:gasket_dec}.

We first establish some preliminary bounds on $\gamma_+^{(i)}$, $i=1,2$. These bounds improve on \eqref{eq:bound_h1_h2}.
\begin{Lem}[Bounds on $\gamma_+^{(i)}$ off self-duality]
	\label{lem:bounds_gamma_i}
	Assuming $x_1< x_2 \leq x_c$, we have the following bounds 
	\begin{equation} \label{eq:bound_h1_h2_sep}
	x_1\gamma_+^{(1)} < \frac{1}{2} \quad \text{and} \quad x_2\gamma_+^{(2)} \leq \frac{1}{2}.
	\end{equation}
	If $x_1 = x_2 < x_c$, we have 
	\begin{equation} \label{eq:bound_h1_h2_sep_bis}
	x_1\gamma_+^{(1)} = x_2\gamma_+^{(2)} < \frac{1}{2}.
	\end{equation}
	In any case, for $x_1\leq x_2 \leq x_c$ and $(x_1,x_2)\neq (x_c,x_c)$, we always have
	\[
	x_1 \gamma_+^{(1)} + x_2 \gamma_+^{(2)} < 1.
	\]
\end{Lem}
\begin{proof}
We first prove our claim \eqref{eq:bound_h1_h2_sep} on colour $1$. Indeed, using the inequality $x_1<x_2$ in \eqref{eq:bound_h1_h2}, we first deduce that $x_1 (\gamma_+^{(1)} + \gamma_+^{(2)}) < 1$. 
Let 
\begin{equation} \label{eq:gen_functions_i}
	\varphi^{(1)}(z) = \sum_{\ell\ge 0} F^{(1)}_\ell z^{\ell+1}
	\quad \text{and} \quad
	\varphi^{(2)}(z) = \sum_{\ell\ge 0} F^{(2)}_\ell z^{\ell+1}.
\end{equation}
By the Vivanti-Pringsheim theorem, the radii of convergence $R^{(1)}$ and $R^{(2)}$ of $\varphi^{(1)}(z)$ and $\varphi^{(2)}(z)$ are $1/\gamma^{(1)}_+$ and $1/\gamma^{(2)}_+$ respectively. Moreover, since $x_2>x_1$, $F^{(2)}_\ell > F^{(1)}_\ell$ and so  $R^{(1)} \ge R^{(2)}$, meaning $\gamma^{(1)}_+ \le \gamma^{(2)}_+$. 
Hence, from $(\gamma_+^{(1)} + \gamma_+^{(2)})x_1  < 1$ we conclude that $2\gamma^{(1)}_+x_1<1$, which proves our claim on colour $1$.

Let us prove our claim \eqref{eq:bound_h1_h2_sep} on colour $2$. Let
\[
\varphi_c(z) = \sum_{\ell\ge 0} F_\ell z^{\ell+1},
\]
where $F_\ell$ is the partition function corresponding to the self-dual parameter $x_c$. Because $x_2\leq x_c$, the radius of convergence $R_c$ of $\varphi_c$ is smaller than $R_2$. Again, by the Vivanti-Pringsheim theorem, we deduce the inequality $\gamma_+^{(2)} \leq \gamma_+$ where $\gamma_+$ refers to the self-dual model. In the latter model, we have shown that, in fact, the formula $\gamma_+ = 1/(2x_c)$ holds (see \cref{prop:determ_gamma}). This shows that $\gamma_+^{(2)} \leq 1/2x_c$. Therefore, $x_2 \gamma_+^{(2)} \leq \tfrac{x_2}{2x_c}$ and we conclude the proof of our claim on colour $i=2$ by the assumption $x_2\leq x_c$.

The proof of \eqref{eq:bound_h1_h2_sep_bis} is easier since we clearly have $x_1\gamma_+^{(1)} = x_2\gamma_+^{(2)}$ by the condition that $x_1=x_2$, and then $x_2\gamma_+^{(2)} < 1/2$ repeating the above paragraph. 
The last claim of \cref{lem:bounds_gamma_i} is a simple consequence of \eqref{eq:bound_h1_h2_sep} and \eqref{eq:bound_h1_h2_sep_bis}.
\end{proof}

We can now complete the proof of \cref{thm:main_offsd}.
\begin{proof}[Proof of \cref{thm:main_offsd}]
Let $i=1,2$ and denote by $j$ the complementary colour. By the gasket decomposition, we may express the partition function $K_{\ell}^{(i)}$ as follows:
\begin{equation} \label{eq:K_ell_gasket}
K_{\ell}^{(i)}
=
n F_{\ell}^{(i)} \sum_{\ell' \geq 0}  A_{\ell,\ell'}^{(i\to j)} F_{\ell'}^{(j)},
\end{equation}
where $A_{\ell,\ell'}^{(i\to j)}$ is the partition function for rings formed of $\ell$ triangles of colour $i$ and $\ell'$ triangles of colour $j$ (see \cref{sss:gasket_dec}). 
We split the proof according to the cases when the parameters are in $\mathcal{S}_q$ or in $\mathcal{R}_q \setminus \mathcal{S}_q$.

\paragraph{Exponential decay in the region $\mathcal{S}_q$.}
We first prove the part of the statement in \cref{thm:main_offsd} concerning exponential decay away from self-duality. To do so, we need to bound \eqref{eq:K_ell_gasket} by an exponential term in $\ell$.
By definition of the cut, the power series $\varphi^{(1)}$ and $\varphi^{(2)}$ in \eqref{eq:gen_functions_i} have radii of convergence $1/\gamma_+^{(1)}$ and $1/\gamma_+^{(2)}$ respectively. This entails that, for any $\varepsilon>0$, the values $\varphi^{(1)}(((1+\varepsilon)\gamma_+^{(1)})^{-1})$ and $\varphi^{(2)}(((1+\varepsilon)\gamma_+^{(2)})^{-1})$ are finite. In particular, the terms 
\[
\frac{F_{\ell}^{(k)}}{((1+\varepsilon)\gamma_+^{(k)})^{\ell}}, \quad k=1,2,
\]
are bounded in $\ell$. Therefore, there exists a constant $C=C(\varepsilon)>0$ such that, for all $\ell\geq 0$, 
\[
F_{\ell}^{(k)} \leq C ((1+\varepsilon)\gamma_+^{(k)})^{\ell}.
\]

Let us plug these bounds into \eqref{eq:K_ell_gasket}.
Then, we get a constant $C=C(\varepsilon)>0$ such that
\begin{equation} \label{eq:K_ell_intermediate}
K_{\ell}^{(i)}
\leq 
C ((1+\varepsilon)\gamma_+^{(i)})^{\ell} \sum_{\ell' \geq 0}  A_{\ell,\ell'}^{(i\to j)} ((1+\varepsilon)\gamma_+^{(j)})^{\ell'}.
\end{equation}
The above sum is a power series that has been introduced in \cite{BBG12}. More specifically, we deduce from equations (3.16)--(3.18) in \cite{BBG12} that 
\begin{equation} \label{eq:ring_part_formula}
\sum_{\ell\geq 1} \sum_{\ell' \geq 0}  A_{\ell,\ell'}^{(i\to j)} r^{\ell} t^{\ell'}
=
\frac{x_i r}{1 - x_i r - x_j t}.
\end{equation}
We are rather interested in the sum over $\ell' \geq 0$ when $\ell$ is fixed and with $t=\gamma_+^{(j)}$, that is the $r^k$-coefficient in the bivariate power series. By expanding in $r$ the ratio on the right-hand side of \eqref{eq:ring_part_formula} at $t=(1+\varepsilon)\gamma_+^{(j)}$, we get that 
\[
\sum_{\ell' \geq 0}  A_{\ell,\ell'}^{(i\to j)} ((1+\varepsilon)\gamma_+^{(j)})^{\ell'}
=
\bigg(\frac{x_i}{1-(1+\varepsilon)x_j\gamma_+^{(j)}}\bigg)^{\ell+1}.
\]
Note that the above expansions make sense for $\varepsilon>0$ small enough since $x_j\gamma_+^{(j)} \leq 1/2$ by \cref{lem:bounds_gamma_i}.
Thus, from \eqref{eq:K_ell_intermediate} we end up with
\begin{equation} \label{eq:K_ell_geom}
K_{\ell}^{(i)}
\leq 
C \bigg(\frac{(1+\varepsilon)x_i\gamma_+^{(i)}}{1-(1+\varepsilon)x_j\gamma_+^{(j)}}\bigg)^{\ell}.
\end{equation}
By \cref{lem:bounds_gamma_i}, we can choose $\varepsilon>0$ small enough so that the above ratio is in $(0,1)$, which concludes the proof of \cref{thm:main_offsd} in the case of $\mathcal{S}_q$.

\paragraph{Polynomial decay at the self-dual point $(s,p_0)= (x_c,p_{\mathrm{sd}}(q))$.}
We now complete the proof of \cref{thm:main_offsd} by showing our last claim about the self-dual regime. We start again from \eqref{eq:K_ell_gasket} which, at the self-dual point, becomes
\begin{equation} \label{eq:K_ell_gasket_sd}
K_{\ell}
=
n F_{\ell} \sum_{\ell' \geq 0}  A_{\ell,\ell'} F_{\ell'}.
\end{equation}
But, at the self-dual point, the partition function is explicitly given by our \cref{thm:F_ell_main} as
\[
F_k := 
\int_{\gamma_-}^{\gamma_+}
	\rho(y) y^{k} \mathrm{d}y,
    \quad k\geq 1,
\]
where $(\gamma_-,\gamma_+,\rho)$ are also explicit. We may then plug this expression into \eqref{eq:K_ell_gasket_sd} and swap the sum and the integral to get
\[
K_{\ell}
=
n F_{\ell} \int_{\gamma_-}^{\gamma_+} \mathrm{d}y \cdot \rho(y) \sum_{\ell' \geq 0}  A_{\ell,\ell'} y^{\ell'}.
\]
Again we can use the general formula \eqref{eq:ring_part_formula} to reduce the expression to
\[
K_{\ell}
=
n F_{\ell} \int_{\gamma_-}^{\gamma_+}  \rho(y) \Big(\frac{x_c}{1-x_c y}\Big)^\ell \mathrm{d}y.
\]
Now recall from \cref{prop:determ_gamma} that $x_c=1/\gamma_+$, so that
\[
K_{\ell}
=
n F_{\ell} \gamma_+^{-\ell} \int_{\gamma_-}^{\gamma_+} \rho(y) (2-y/\gamma_+)^{-\ell} \mathrm{d}y.
\]
We now use the change of variables $u:= \ell(\gamma_+-y)$ in the above integral, which yields
\[
K_{\ell}
=
n F_{\ell} \gamma_+^{-\ell} \int_{0}^{(\gamma_+-\gamma_-)\ell} \rho(\gamma_+-u/\ell) \Big(1+\frac{u}{\gamma_+ \ell}\Big)^{-\ell} \frac{\mathrm{d}u}{\ell}.
\]
Finally, using the expression for $\rho$ in \cref{thm:F_ell_main} and dominated convergence, we see that there exists a constant $C>0$ such that
\[
K_{\ell}
\sim
C F_{\ell} \frac{\gamma_+^{-\ell}}{\ell^{2-\theta}}
\quad \text{as } \ell\to\infty.
\]
We conclude from \cref{cor:asymp_F_ell_main} that $K_\ell$ has the polynomial decay claimed in \cref{thm:main_offsd}.
\end{proof}


\appendix

\section{Computation of inverse Fourier transform}

\label{sec:appendix}

Here we verify that the formula computing the inverse Fourier transform of $R_+$ in Section \ref{sec:resolvents} (crucial in the proof of Theorem \ref{thm:F_ell_main}) is as claimed in \eqref{eq:r+explicit}. Recall that 
we know from the Wiener--Hopf argument that the Fourier transform $R_+ (\omega) = \textstyle\int_{\mathbb{R}} \mathrm{e}^{ i \omega v} r_+ (v) \mathrm{d} v$ satisfies
\begin{align*}
R_+(\omega) &= 
    - \frac{4\sqrt{2}}{\pi \theta} \sin\left( \frac{\pi\theta}{2}\right) 
    \frac{\Gamma\left( \frac{3+2\theta-i\omega}{4} \right) \Gamma\left( \frac{3-2\theta-i\omega}{4} \right)}{2^{i\omega/2} \Gamma\left(\frac{1-i\omega}{2}\right) (\omega+i)(\omega+3i)}\\
    & =      \frac{\sqrt{2}}{\pi \theta} \sin\left( \frac{\pi\theta}{2}\right) 
    \frac{\Gamma\left( \frac{3+2\theta-i\omega}{4} \right) \Gamma\left( \frac{3-2\theta-i\omega}{4} \right)}{2^{i\omega/2}\Gamma ( \tfrac{5 - i \omega}{2})}.
\end{align*}
where we used properties of the Gamma function to simplify the denominator of the fraction. 

It is easier to start from the answer, i.e., to consider the function 
$$
s_+ (v) : = \frac{(\gamma_+-\gamma_-)}{\sqrt{2} \pi\theta} \mathrm{e}^{-3v} \left( (\mathrm{e}^{2v}+\sqrt{\mathrm{e}^{4v}-1})^{\theta} - (\mathrm{e}^{2v}-\sqrt{\mathrm{e}^{4v}-1})^{\theta} \right), \quad v >0,
$$
compute its Fourier transform $S_+(\omega) = \int_{0}^\infty \mathrm{e}^{ i \omega v} s_+ (v) \mathrm{d} v$ and verify that $S_+(\omega) = R_+(\omega)$. (We will verify this for $\omega \in \mathcal{H}_+$, i.e., where $R_+$ is well defined). Let us write $c =\tfrac{\gamma_+-\gamma_-}{\sqrt{2} \pi\theta}$. 

Let us make the change of variables $\mathrm{e}^{2v} = \cosh(u)$, so $2 \cosh(u) \mathrm{d} v = \sinh(u) \mathrm{d} u$ and $u$ integrates from $0$ to $\infty$. Furthermore, $\sqrt{\mathrm{e}^{4v} -1} = \sqrt{\cosh(u)^2 -1} = \sinh (u)$, hence 
$$
S_+(v) = c \int_0^\infty\cosh(u)^{\tfrac{i \omega}{2}-\tfrac{3}{2}} \left( ( \cosh u + \sinh u)^\theta - ( \cosh u - \sinh u)^\theta\right) \frac{\sinh u}{2 \cosh u} \mathrm{d} u.
$$
Note also that $( \cosh u + \sinh u)^\theta - ( \cosh u - \sinh u)^\theta = \mathrm{e}^{ u \theta} - \mathrm{e}^{ - u\theta} = 2\sinh (\theta u)$, so
$$
S_+(v) = c \int_0^\infty (\cosh u)^{ \tfrac{i \omega}{2}-\tfrac{5}{2}} \sinh (\theta u) \sinh (u) \mathrm{d} u.
$$
Now, recall that 
$$
\sinh (\theta u ) \sinh (u) = \frac12\large( \cosh ( u (\theta +1)) - \cosh ( u ( \theta -1))\large),
$$
so that 
\begin{equation}\label{eq:S+1}
S_+(v) = \frac{c}{2} \int_0^\infty(\cosh u)^{ \tfrac{i \omega}{2}-\tfrac{5}{2}}  \cosh ( u ( \theta + 1)) \mathrm{d} u - \frac{c}{2} \int_0^\infty(\cosh u)^{ \tfrac{i \omega}{2}-\tfrac{5}{2}}  \cosh ( u ( \theta - 1)) \mathrm{d} u.
\end{equation}
Such integrals can be explicitly computed in terms of Beta functions: indeed, if $a,b \in \mathbb{C}$ with $\Re (a) > | \Re (b)|$ then (see \cite[(5.12.7)]{DLMF})
\begin{equation}\label{eq:Beta}
\int_0^\infty \frac{\cosh ( 2bt))}{ (\cosh t)^{2a}} \mathrm{d} t = 4^{a-1} B( a+b, a-b),
\end{equation}
where 
$$
B(x,y) = \frac{\Gamma(x) \Gamma(y)}{\Gamma(x+y)}.
$$
We thus use \eqref{eq:Beta} with $2a = \alpha =  5/2 - i \omega /2 $ and $2b = \theta \pm 1$, so that the condition $\Re (a) > |\Re(b)|$ is fulfilled when $\omega \in \mathcal{H}_+$ (recall that $\theta \in (0,1/2)$). We obtain:
\begin{align*}
S_+(\omega) &= \frac{c}{2} 2^{\tfrac{5}{2} - \tfrac{i\omega}{2} - 2} \left( \frac{\Gamma(\tfrac{\alpha + \theta +1}{2}) \Gamma ( \tfrac{\alpha - \theta -1}{2}) }{\Gamma(\alpha)} - \frac{\Gamma(\tfrac{\alpha + \theta -1}{2} \Gamma ( \tfrac{\alpha - \theta +1}{2})}{\Gamma(\alpha)}\right)\\
& = \frac{c 2^{ -\tfrac{1}{2} - \tfrac{i \omega}{2} } }{\Gamma(\alpha)} \left( \Gamma(\tfrac{\alpha + \theta +1}{2} )\Gamma ( \tfrac{\alpha - \theta -1}{2})  - \Gamma(\tfrac{\alpha + \theta -1}{2} \Gamma ( \tfrac{\alpha - \theta +1}{2})
\right).
\end{align*}

We now simplify the above expression using standard properties of the Gamma function. We first use the relation $\Gamma(x+1) = x\Gamma(x)$ twice, with $x=\tfrac{\alpha+\theta-1}{2}$ and $x = \tfrac{\alpha+\theta-1}{2}$ respectively. This gives 
\begin{align*}
S_+(\omega) 
&= 
\frac{c}{\sqrt{2} \Gamma(\alpha)} 2^{-i\omega/2} \Gamma\Big(\frac{\alpha+\theta-1}{2}\Big) \Gamma\Big(\frac{\alpha-\theta-1}{2}\Big) \Big( \frac{\alpha+\theta-1}{2} - \frac{\alpha-\theta-1}{2} \Big) \\
&=
\frac{c\theta}{\sqrt{2}\Gamma(\alpha)} 2^{-i\omega/2} \Gamma\Big(\frac{\alpha+\theta-1}{2}\Big) \Gamma\Big(\frac{\alpha-\theta-1}{2}\Big).
\end{align*}
Recalling from \eqref{eq:gamma+} the value of $(\gamma_+ - \gamma_-) = \tfrac{2^{3/2}}{\theta} \sin( \pi \theta/2)$, we conclude that $S_+(\omega)$ has the same value as $R_+(\omega)$ as computed at the beginning of the appendix. This concludes the proof of \eqref{eq:r+explicit}.


\addcontentsline{toc}{section}{References}
\bibliographystyle{alpha}
\bibliography{biblio}

\end{document}